\documentclass{article}

\usepackage{apacite}

\usepackage{amssymb}
\usepackage[mathscr]{eucal}
\usepackage{amsmath,amsthm,latexsym,bbm}
\usepackage{setspace}
\usepackage[margin=1in]{geometry}
\usepackage{url}

\newcommand{\NN}{\mathbb{N}}

\newcommand{\RR}{\mathbb{R}}

\newcommand{\bA}{{\boldsymbol{A}}}

\newcommand{\bcB}{{\boldsymbol{\cB}}}

\newcommand{\bC}{{\boldsymbol{C}}}

\newcommand{\bD}{{\boldsymbol{D}}}
\newcommand{\be}{{\boldsymbol{e}}}
\newcommand{\bE}{{\boldsymbol{E}}}
\newcommand{\bI}{{\boldsymbol{I}}}

\newcommand{\bQ}{{\boldsymbol{Q}}}
\newcommand{\tbQ}{\widetilde{\bQ}}

\newcommand{\tX}{\widetilde{X}}

\newcommand{\bu}{{\boldsymbol{u}}}

\newcommand{\bv}{{\boldsymbol{v}}}

\newcommand{\bX}{{\boldsymbol{X}}}
\newcommand{\by}{{\boldsymbol{y}}}
\newcommand{\bY}{{\boldsymbol{Y}}}

\newcommand{\bZ}{{\boldsymbol{Z}}}
\newcommand{\bU}{{\boldsymbol{U}}}
\newcommand{\tbU}{{\widetilde{\bU}}}

\newcommand{\balpha}{{\boldsymbol{\alpha}}}

\newcommand{\bvare}{{\boldsymbol{\vare}}}

\newcommand{\bvarphi}{{\boldsymbol{\varphi}}}
\newcommand{\boldeta}{{\boldsymbol{\eta}}}
\newcommand{\btheta}{{\boldsymbol{\theta}}}
\newcommand{\bxi}{{\boldsymbol{\xi}}}

\newcommand{\bzero}{{\boldsymbol{0}}}

\newcommand{\cB}{{\mathcal B}}

\newcommand{\cD}{{\mathcal D}}

\newcommand{\cF}{{\mathcal F}}

\newcommand{\cM}{{\mathcal M}}
\newcommand{\bcM}{\boldsymbol{\cM}}

\newcommand{\cS}{{\mathcal S}}

\newcommand{\cZ}{{\mathcal Z}}
\newcommand{\cW}{{\mathcal W}}
\newcommand{\bcW}{\boldsymbol{\cW}}

\newcommand{\bcZ}{\boldsymbol{\cZ}}

\newcommand{\ee}{\mathrm{e}}

\newcommand{\oo}{\mathrm{o}}
\newcommand{\OO}{\mathrm{O}}

\newcommand{\EE}{\operatorname{E}}
\newcommand{\PP}{\operatorname{P}}
\newcommand{\var}{\operatorname{Var}}
\newcommand{\cov}{\operatorname{Cov}}

\newcommand{\hcM}{\widehat{\cM}}
\newcommand{\hbcM}{\widehat{\bcM}}

\newcommand{\hM}{\widehat{M}}

\newcommand{\hbI}{\widehat{\bI}}
\newcommand{\hbZ}{\widehat{\bZ}}

\newcommand{\halpha}{\widehat{\alpha}}
\newcommand{\hbalpha}{\widehat{\balpha}}

\newcommand{\hbtheta}{\widehat{\btheta}}

\newcommand{\htau}{\widehat{\tau}}

\newcommand{\tbalpha}{\widetilde{\balpha}}

\newcommand{\tbtheta}{\widetilde{\btheta}}
\newcommand{\tmu}{\widetilde{\mu}}
\newcommand{\hmu}{\widehat{\mu}}

\newcommand{\tbX}{\widetilde{\bX}}

\newcommand{\tU}{\widetilde{U}}

\newcommand{\vare}{\varepsilon}

\newcommand{\mxk}{
\begin{bmatrix}
\bX_{k-1}
\\
1
\end{bmatrix}
}

\newcommand{\mx}[1]{
\begin{bmatrix}
#1
\\
1
\end{bmatrix}
}

\newcommand{\nrho}{{\lfloor n \rho \rfloor}}

\newcommand{\vnorm}[1]{\left\lVert #1 \right\rVert}

\renewcommand{\mid}{\,|\,}

\renewcommand{\leq}{\leqslant}
\renewcommand{\geq}{\geqslant}

\newcommand{\stoch}{\stackrel{\PP}{\longrightarrow}}
\newcommand{\distr}{\stackrel{\cD}{\longrightarrow}}
\newcommand{\distre}{\stackrel{\cD}{=}}

\newcommand{\as}{\stackrel{{\mathrm{a.s.}}}{\longrightarrow}}

\newcommand{\nt}{{\lfloor nt\rfloor}}

\begin{document}
\numberwithin{equation}{section}
\theoremstyle{plain}
\newtheorem{Thm}{Theorem}[section]
\newtheorem{Lem}[Thm]{Lemma}
\newtheorem{Pro}[Thm]{Proposition}
\newtheorem{Cor}[Thm]{Corollary}

\theoremstyle{definition}
\newtheorem*{defi}{Definition}
\newtheorem*{remark}{Remark}

\newcommand{\Q}{\textbf{???}}

\begin{center}
{\Large\bf Change detection in INAR($p$) processes \\
 against various alternative hypotheses}
\\[18pt]
{\large\sc Gyula Pap and Tam\'as T. Szab\'o}

{\large Bolyai Institute, University of Szeged, Szeged, Hungary}

Address of correspondence: Tam\'as T. Szab\'o, Bolyai Int\'ezet, Aradi v\'ertan\'uk tere 1, Szeged, H-6720, Hungary; E-mail: tszabo@math.u-szeged.hu

\end{center}

\begin{abstract}
\emph{Change in the coefficients or in the mean of the innovation of an INAR($p$) process is a sign of disturbance that is important to detect. 
The proposed methods can test for change in any one of these quantities
 separately, or in any collection of them.
They make both one-sided and two-sided tests possible, furthermore,
 they can be used to test against the `epidemic' alternative. 
The tests are based on a CUSUM process using CLS
 estimators of the parameters.
Under the one-sided and two-sided alternatives consistency of the tests is proved and the properties of the change-point estimator are also explored.}
\end{abstract}

\section{Introduction}\label{intro}

A time-inhomogeneous INAR($p$) process is a sequence \ $(X_k)_{k \geq -p+1}$ \ given by
 \begin{align}\label{GWI}
  X_k = \sum_{j=1}^{X_{k-1}} \xi_{1, k, j} + \cdots
        + \sum_{j=1}^{X_{k-p}} \xi_{p, k, j} + \vare_k, \qquad k \in \NN ,
 \end{align}
 where \ $\{ \vare_k : k \in \NN \}$ \ is a sequence of independent non-negative
 integer-valued random variables, for each \ $k \in \NN$ \ and
 \ $i \in \{ 1, \dots, p\}$ \ the sequence \ $\{ \xi_{i, k, j} : j \in \NN \}$
 \ is a sequence of i.i.d.~Bernoulli random variables with mean
 \ $\alpha_{i,k}$ \ such that these sequences are mutually independent and
 independent of the sequence \ $\{ \vare_k : k \in \NN \}$, \ and
 \ $X_0$, \dots, $X_{-p+1}$ \ are non-negative integer-valued random
 variables independent of the sequences \ $\{ \xi_{i, k, j} : j \in \NN \}$,
 \ $k \in \NN$, \ $i \in \{ 1, \dots, p\}$, \ and
 \ $\{ \vare_k : k \in \NN \}$. The numbers \ $\alpha_{i,k}$ \ are called coefficients,
 and we will refer to \ $\vare_1, \ \vare_2, \ldots \ $ as the innovations.
\ Time-homogeneous INAR($p$) processes have a number of applications, which are summarized, e.g., in \shortciteA{BarczAsymp_10}.

The reason that we initially define our process as time-inhomogeneous is that we would like to test for a change in the parameters, therefore we have to allow them to vary over time. In the proofs, however, a majority of the results will be based upon the properties of time-homogeneous INAR($p$) processes.

In applications the time series underlying an observation is usually supposed to be homogeneous, and this can lead to false conclusions if the parameters have changed during the time of the observation. Therefore testing for a change in the parameters has been an important question.
The monograph of \citeA{CsHorv_97} gives an excellent overview of the subject for many times series models, and in their own context, their main results are stronger than those of this paper---e.g., in place of Theorem \ref{HA_mu_2} they usually prove convergence in distribution under a variety of conditions.
Their methods, however, cannot be directly used to handle INAR($p$) processes---they routinely involve a Taylor expansion of the likelihood function, which we cannot do because we are not assuming that the innovation distribution comes from a parametric family.
Furthermore, the autoregressive processes investigated in \citeA{CsHorv_97} have the usual type of autoregression---i.e., the next element of the time series is a measurable function of the previous elements, plus an independent innovation.
In our case, however, the next element depends stochastically on the previous elements as well. Therefore, it is more fruitful to consider our model as a special multitype branching process.
This will be our approach as well, but we remark that much of the motivation for the paper is due to \citeA{CsHorv_97} and \citeA{Gomb_08}.
We note that our Theorem \ref{theta} and Theorem \ref{HA_mu_1} are weaker than the equivalent versions of Theorems 1 and 2 in \citeA{Gomb_08}---this is due to our model being more complicated. In our continued research we plan to improve our results in this direction.
Similarly, in place of Theorem \ref{HA_mu_2} we hope to show that the difference $\tau_n - \tau$ converges in distribution, as it was shown in \citeA{CsHorv_97} for the simpler models. Finally, we remark that it is indeed proper to take motivation from these sources because our process resembles an AR($p$) process in its covariance structure.

Change detection methods for INAR($p$) processes in general (i.e., with no prespecified innovation distribution) have only been proposed in a few papers---we refer to \citeA{KLee_09} especially, where the authors give a test statistics similar to ours for a more general model. However, no result is available under the alternative hypothesis and the asymptotics of the change-point estimator are not given---we will give results to these questions which strengthen the theoretical foundations of the method considerably.

The test has already been published in \citeA{TSz_11}. However, the referred paper only states the main result but does not provide the proof. Also, it contains no results under the alternative hypothesis either, which is the main result of this paper. In the appendix we will also provide the missing (although rather standard) proof in \citeA{TSz_11}.

Now we proceed with the formulation of the statistical problem. We assume \ $\mu_k := \EE(\vare_k) < \infty$ \ and
 \ $0 < \sigma^2_k := \var(\vare_k) < \infty$.

Write the parameter vectors as
 \begin{align*}
  \begin{bmatrix}
   \alpha_{1,k} \\
   \vdots \\
   \alpha_{p,k} \\
   \mu_k
  \end{bmatrix}
     =: \btheta_k ,
 \end{align*}
 and let us choose a subset $PV$ of $\{1,2,\ldots,p+1\}$ such that
 \begin{equation}\label{eq:PV}
 PV = \{i_1,i_2,\ldots,i_\ell\} \text{ for some } \ell > 0, \quad i_1 < i_2 < \ldots < i_\ell.
 \end{equation}
 Also, we can write
 \[
 NV := PV^\complement = \{j_1,j_2,\ldots,j_{p+1-\ell}\}, \quad j_1 < j_2 < \ldots < j_{p+1-\ell},
 \] 
 where $H^\complement$ denotes the complement of a set.
 Let us now define
 \[
 \bvarphi_k:=(\theta_k^{(i_1)},\theta_k^{(i_2)}, \ldots, \theta_k^{(i_\ell)})^\top, \quad \boldeta_k:=(\theta_k^{(j_1)},\theta_k^{(j_2)}, \ldots, \theta_k^{(j_{p+1-\ell})})^\top.
 \]
 
 The vector \ $\bvarphi_k$ \ is the parameter vector of interest and \ $\boldeta_k$ \ is the 'nuisance' parameter vector.
For a fixed number of observations $n$ we want to test the null hypothesis
 \begin{align*}
  \mathrm{H_0}: \qquad 
  \text{$\vare_1, \dots, \vare_n$ \ are identically distributed \ and
        \ $\btheta_1 = \btheta_2 = \dots = \btheta_n$}
 \end{align*}
 against the alternative
 \begin{align*}
  \mathrm{H}_\mathrm{A}: \qquad 
  &\text{there is an integer \ $\tau \in \{1, \dots, n-1\}$ \ such that} \\
  &\text{$\bvarphi_1 = \dots = \bvarphi_\tau \ne \bvarphi_{\tau + 1} = \dots
         = \bvarphi_n$
         \ but \ $\boldeta_1 = \dots = \boldeta_n$,} \\
  &\text{$\vare_1, \dots, \vare_\tau$ \ are identically distributed,} \\
  &\text{and \ $\vare_{\tau + 1}, \dots, \vare_n$ \ are identically distributed.}
 \end{align*}
The main difficulty is that the change is not in a location or scaling
 parameter as in most of the change detection models, and standard techniques
 based on likelihood are not applicable.

Under the null hypothesis \ $\mathrm{H_0}$ \ we have
 \begin{align*}%\label{H_0}
  \begin{bmatrix}
   \alpha_{1,1} \\
   \vdots \\
   \alpha_{p,1} \\
   \mu_1
  \end{bmatrix}
  = \dots = \begin{bmatrix}
             \alpha_{1,n} \\
             \vdots \\
             \alpha_{p,n} \\
             \mu_n
             \end{bmatrix}
  =: \begin{bmatrix}
      \alpha_1 \\
      \vdots \\
      \alpha_p \\
      \mu
     \end{bmatrix}
  =: \begin{bmatrix}
      \balpha \\
      \mu
     \end{bmatrix}
  =: \btheta ,
\qquad
\sigma^2_1=\ldots=\sigma^2_n=:\sigma^2,
 \end{align*}
 and the process \eqref{GWI} is referred to as stable, unstable or explosive
 whenever \ $\alpha_1 +\cdots + \alpha_p < 1$,
 \ $\alpha_1 +\cdots + \alpha_p = 1$ \ or \ $\alpha_1 +\cdots + \alpha_p > 1$,
 \ respectively. Basic differences between the three types are summarized in \citeA{BarczAsymp_10}.
We will study only the stable case \ $\alpha_1 +\cdots + \alpha_p < 1$
 \ with nonvanishing innovation (i.e.~with \ $\mu > 0$), when the Markov chain
 \ $(\bX_k)_{k \geq 0}$ \ given by
 \[
   \bX_k
   := \begin{bmatrix}
       X_k \\
       \vdots \\
       X_{k-p+1}
      \end{bmatrix} , \qquad
   k \geq 0 ,
 \]
 has a nondegenerate, unique stationary distribution (see e.g.~\citeA{Qu_70}).
 Under the alternative hypothesis we will require that the process be stable both before and after the change.
 This assumption is central to our proofs, because if the process is not stable, then it is not ergodic either and we cannot use any of the methods outlined below.
 For unstable and explosive processes completely different tools should be developed.
 
 The structure of the paper is the following. In the second section we propose a CUSUM-like test process similar to that of \citeA{KLee_09}, in the third section we compute its limiting distribution and, based on the process, we introduce both one-sided and two-sided tests. For simulations on the power of these tests we refer the reader to \citeA{TSz_11}. Our main results under the alternative hypothesis are stated in Section 4 and a real data illustration is given in Section 5. Appendix A contains detailed proofs under the null hypothesis, while Appendix B contains the proofs under the alternative hypothesis. The paper is concluded by Appendix C, which contains some necessary technical lemmas and calculations. Unless otherwise noted, we understand convergence as \ $n \to \infty.$ We will denote the set of positive integers by $\NN$ and the set of nonnegative integers by $\NN_0$.

\section{Parameter estimation and the construction of our test process}\label{sec:parest}
\subsection{Parameter estimates}
The first step in parameter estimation is introducing the sequence of martingale differences $(M_k)_{k \in \NN}$:
\begin{equation}\label{Mdefi}
M_k = X_{k} - \EE(X_k|\cF_{k-1}) = X_k - \balpha^\top \bX_{k-1} - \mu, \qquad k \in \NN,
\end{equation}
where $(\cF_n)_{n \in \NN}$ is the natural filtration. The conditional least squares estimators of the parameters, first introduced by \citeA{KlNe_78}, can be calculated by minimizing the sum of squares
\[
R_n(\alpha_1, \ldots, \alpha_p,\mu)
 := \frac{1}{2}\sum_{k=1}^{n}M_k^2
 = \frac{1}{2}\sum_{k=1}^{n}(X_k - \balpha^\top \bX_{k-1} - \mu)^2
\]
with respect to $\alpha_1, \ldots, \alpha_p, \mu$. 
With a reasoning completely analogous to that of 3.1 Lemma and 3.1 Proposition in \citeA{BarczPrimitive_12} we can show that $R_n$ has a unique minimum given by \eqref{CLSests} whenever $\bQ_n$ is invertible.
Now, we will show in \ref{sec:invert} that this is true with an asymptotic probability of 1, therefore the parameter estimates exist and are unique with an asymptotic probability of 1.
As all our results our asymptotic, this will be sufficient for the purposes of the paper.
\begin{align}\label{CLSests}
  \hbtheta_n :=
  \begin{bmatrix}
   \hbalpha_n \\
   \hmu_n
  \end{bmatrix}
  := \bQ_{n}^{-1}
     \sum_{k=1}^n
      X_k
      \begin{bmatrix}
       \bX_{k-1} \\
       1
      \end{bmatrix} , \qquad
  \bQ_{n} := \sum_{k=1}^n
               \begin{bmatrix}
                \bX_{k-1} \\
                1
               \end{bmatrix}
               \begin{bmatrix}
                \bX_{k-1} \\
                1
               \end{bmatrix}^\top .
 \end{align}
We note that this estimator is strongly consistent under the null hypothesis (\citeNP{DuLi_91}), and the estimation procedure itself supposes that the null hypothesis is valid; however, the calculations can be carried out under the alternative hypothesis as well. Under the alternative hypothesis the weak limit of \ $\hbtheta_n$ \ is given in Lemma \ref{lem:asymperror}.
 Replacing the parameters by their estimates in $M_k$ we obtain $\hM_k^{(n)}$, i.e.,
 \begin{align}\label{hM_n}
  \hM_k^{(n)} := X_k - 
					 \hbtheta_n^\top
					 \begin{bmatrix}
                      \bX_{k-1} \\
                      1
                     \end{bmatrix}
                      .
 \end{align}
Although not a parameter in which we are looking for change, the estimate of the variance of the innovation $\sigma^2$ will also appear in our test process, therefore we have to provide an estimator for it. To do this, we introduce
\[
N_k = M_k^2 - E(M_k^2|\cF_{k-1}) = M_k^2 - \alpha_1(1-\alpha_1)X_{k-1} - \cdots - \alpha_p(1-\alpha_p) X_{k-p} - \sigma^2 , \qquad k \geq 0.
\]
Minimizing $\sum_{k=1}^n N_k^2$ with respect to $\sigma^2$ we
obtain the conditional least squares estimate
\begin{equation}\label{eq:barsigma}
\overline{\sigma}^2_n
=
-\frac{1}{n}
\sum_{k=1}^n
(M_k^2 - \alpha_1 (1-\alpha_1) X_{k-1} - \cdots - \alpha_p(1-\alpha_p) X_{k-p}).
\end{equation}
However, in this estimate the true parameters are still present. The estimate that we will use is given by replacing the $\alpha$ coefficients and $\mu$ both in the formula and in
 $M_k^2$ by their estimates:
\begin{equation}\label{eq:hatsigma}
\widehat{\sigma}^2_n
=
-\frac{1}{n}
\sum_{k=1}^n
\left(\left(\hM_k^{(n)}\right)^2 - \halpha_1^{(n)}\left(1-\halpha_1^{(n)}\right) X_{k-1} - \cdots - \halpha_p^{(n)}\left(1-\halpha_p^{(n)}\right) X_{k-p}\right).
\end{equation}

\subsection{Construction of the test process}

We use a formal analogy between the INAR($p$) process and the well-known AR($p$) process
 (\citeNP{Venk_82}) to obtain analogues of score vector and information
 quantities as in \citeA{Gomb_08}.
We briefly recall the motivation of the test process as given in \citeA{TSz_11}. We briefly recall the motivation of the test process as given in \citeA{TSz_11}. Due to the martingale central limit theorem,
\[
\left(\frac{1}{\sqrt{n}}\sum_{k=1}^{\lceil nt \rceil} M_k\right)_{t \in [0,1]} \distr \left(\sqrt{c}{\mathcal W}_t\right)_{t \in [0,1]},
\]
where $c$ is a constant depending on $\btheta$ and $\sigma^2$, and $({\mathcal W}_t)_{0 \leq t \leq 1}$ is a standard Brownian motion. Therefore, by a rough approximation
\[
\left(M_1, \ldots, M_n\right) \sim N(0, c E_n),
\]
where $E_n$ is the $n \times n$ identity matrix. 
The approximate likelihood function is
\[
\frac{1}{(2 \pi c)^{n/2}} \exp\left\{-\frac{1}{2c^2} \sum_{k=1}^{n}M_k^2\right\}.
\]
We will take the derivative of the log-likelihood function and work with that quantity. The first term will be regarded as constant. This is a simplification because $c$ actually depends on the parameters but taking this into account leads to calculations that are difficult to handle. Also, we will not take into account the constant factor before the sum of the $M_k$ but will rather work with the analogue of the information matrix. Therefore, we consider the following analogue of the loglikelihood function:
\[
R_n(\alpha_1,\ldots,\alpha_p,\mu)=-2^{-1}\sum_{k=1}^n M_k^2.
\]
Note that this is the sum that we had to minimize for CLS estimations.
The role of the score vector will be played by
\[
- \nabla R_k(\hbtheta_n) =
\sum_{j=1}^{k}
\widehat{M}_j^{(n)}
\begin{bmatrix}
\bX_{j-1} \\
1 
\end{bmatrix}
.\]
The information matrix $\bI_n$ is defined by
\[
\bI_n
  := \sum_{k=1}^n
      \EE\bigl[ \{\nabla R_k(\btheta) - \nabla R_{k-1}(\btheta)\}
                      \{\nabla R_k(\btheta) - \nabla R_{k-1}(\btheta)\}^\top 
                      \mid \cF_{k-1} \bigr]
                      =
                      \sum_{k=1}^n
                      (\balpha^\top_2 \bX_{k-1} + \sigma^2) \!
                      \mx{\bX_{k-1}} \! \! \mx{\bX_{k-1}}^\top
                      ,
\]
with
\begin{equation}\label{eq:balpha2defi}
\balpha_2:=[\alpha_1(1-\alpha_1), \ldots, \alpha_p(1-\alpha_p)]^\top.
\end{equation}
The estimate $\hbI_n$ is defined by replacing in $\bI_n$ the variance $\sigma^2$ and all the parameters in $\btheta$ by their CLS estimates. This leads to the $p+1$-dimensional test process \ $(\hbcM_n(t))_{0 \leq t \leq 1}$
 \ given by
 \begin{align}\label{hcMn}
  \hbcM_n(t)
  := \hbI_n^{-1/2}
     \sum_{k=1}^\nt
      \hM_k^{(n)}
      \begin{bmatrix}
       \bX_{k-1} \\
       1
      \end{bmatrix}.
 \end{align}
Note that the process \ $(\hbcM_n(t))_{0 \leq t \leq 1}$ \ can also be written in
 the CUSUM form
 \begin{align*}
  \hbcM_n(t)
  &= \hbI_n^{-1/2}
     \left( \sum_{k=1}^\nt
             X_k
             \begin{bmatrix}
              \bX_{k-1} \\
              1
             \end{bmatrix}
            - \bQ_{\nt} \, \bQ_{n}^{-1}
              \sum_{k=1}^n
               X_k
               \begin{bmatrix}
                \bX_{k-1} \\
                1
               \end{bmatrix} \right) \\
  &= \hbI_n^{-1/2} \, \bQ_{\nt}
     \left( \begin{bmatrix}
             \hbalpha_{\nt} \\
             \hmu_{\nt}
            \end{bmatrix}
            - \begin{bmatrix}
               \hbalpha_n \\
               \hmu_n
              \end{bmatrix} \right)
   = \hbI_n^{-1/2} \, \bQ_{\nt} \left( \hbtheta_{\nt} - \hbtheta_n \right) .    
 \end{align*} 
This is close to the CUSUM processes used by \citeA{KLee_09}. In this paper the authors investigated a change in random coefficient INAR($p$) processes based on both maximum-likelihood and conditional least squares estimators. A general convergence theorem was proved for a test process similar to ours based on a broader class of estimators. However, the authors did not state any results under the alternative hypothesis. Our Theorem \ref{theta} is analogous to Theorem 1 of \citeA{KLee_09}. However, while their model is more general than ours, our Theorem \ref{theta} is not a consequence of their Theorem 1, even though the applied methods are largely identical.

\section{Testing procedures}\label{tests}
\begin{defi}
A time-homogeneous INAR($p$) process $(X_k)_{k\geq -p+1}$ is said to satisfy
 condition $\textbf{C}_0$, if $\EE(X_0^6) < \infty$, \dots,
 $\EE(X_{-p+1}^6) < \infty$, \ $\EE(\vare_1^6) < \infty$,
 \ $\alpha_1 + \cdots + \alpha_p < 1$, \ $\mu > 0$ \ all hold for it, and if, furthermore,  \ $\alpha_1 + \cdots + \alpha_p > 0$ \ or \ $\sigma^2 > 0$.
 
 An INAR($p$) process $(X_k)_{k\geq -p+1}$ which satisfies $\mathrm{H_A}$ is said
 to satisfy condition $\textbf{C}_\text{A}$, if
 \ $(X_k)_{-p+1 \leq k \leq\tau}$ \ and \ $(X_k)_{k\geq\tau+1}$ satisfy condition $\textbf{C}_0$.
\end{defi}
Under the null hypothesis we have the following result, which allows the construction of various test statistics.
\begin{Thm}\label{theta}
If $(X_k)_{k\geq -p+1}$ satisfies $\mathrm{H_0}$ and condition
 $\textup{\textbf{C}}_0$, then
 \begin{align*}
  \hbcM_n \distr \bcB \qquad
  \text{as \ $n \to \infty$,}
 \end{align*}
 where \ $(\bcB(t))_{0 \leq t \leq 1}$ \ is a $(p+1)$-dimensional standard Brownian
 bridge, and \ $\distr$ \ denotes convergence in distribution in the
 Skorokhod space \ $D([0,1])$.
\end{Thm}
 
By the continuous mapping theorem we obtain the following corollary.

\begin{Cor}\label{sup}
Under the assumptions of Theorem \ref{theta} we have
 \begin{align}
  \sup_{0 \leq t \leq 1} \hcM_n^{(i)}(t)
  &\distr \sup_{0 \leq t \leq 1} B(t) , \label{wsup} \\
  \inf_{0 \leq t \leq 1} \hcM_n^{(i)}(t)
    &\distr \inf_{0 \leq t \leq 1} B(t) , \\
  \sup_{0 \leq t \leq 1} |\hcM_n^{(i)}(t)|
  &\distr \sup_{0 \leq t \leq 1} |B(t)| , \label{wabssup} \\
  \sup_{0 \leq t \leq 1} \hcM_n^{(i)}(t)
  - \inf_{0 \leq t \leq 1} \hcM_n^{(i)}(t)
  &\distr \sup_{0 \leq t \leq 1} B(t) - \inf_{0 \leq t \leq 1} B(t)
  \label{wsup-min}
 \end{align}
 as \ $n \to \infty$, \ where \ $(\hcM_n^{(i)}(t))_{0 \leq t \leq 1}$,
 \ $i=1,\ldots,p+1$, \ denotes the components of
 \ $(\hbcM_n(t))_{0 \leq t \leq 1}$, \ and \ $(B(t))_{0 \leq t \leq 1}$ \ is a Brownian
 bridge.
\end{Cor}

Theorem \ref{theta} is proved in Appendix \ref{app_B}.

Since \ $(\bcB(t))_{0 \leq t \leq 1}$ \ in Theorem \ref{theta} has independent
 components, we need to define the tests component-wise only. 
For simultaneous test-for-change in \ $d$ \ parameters, to have an overall
 level of significance \ $\alpha$, \ we use
 \ $\alpha^* := 1 - (1 - \alpha)^{1/d}$ \ for each component. We can test for change in a single component, $\theta^{(i)}, \ i \in PV$ (with $\theta^{(i)}=\alpha_i$ for $i=1,\ 2,\  \ldots,\ p$ and $\theta^{(p+1)}=\mu$, according to the definition of $\btheta$) in the following way:

Three different tests can be constructed:

\textbf{Test 1 (one-sided)}: If
 \begin{align*}
  \sup_{0 \leq t \leq 1} \hcM_n^{(i)}(t) \geq C_1(\alpha^*) \quad \text{or} \quad \inf_{0 \leq t \leq 1} \hcM_n^{(i)}(t) \leq - C_1(\alpha^*),
 \end{align*}
 then we conclude that there was a downward or upward change in parameter \ $\theta^{(i)}$ (respectively) along the sequence \ $X_0$, $X_1$, \dots, $X_n$.
 
% Revision up to this point 1203220938 

\textbf{Test 2 (two-sided)}: If
 \begin{align*}
  \sup_{0 \leq t \leq 1} |\hcM_n^{(i)}(t)| \geq C_2(\alpha^*) ,
 \end{align*}
 then we conclude that there was a change in parameter \ $\theta^{(i)}$
 along the sequence \ $X_0$, $X_1$, \dots, $X_n$.

The third test is designed against the so-called epidemic alternative
 (see \citeNP{CsHorv_97}, 1.7.4) and is included for the sake of completeness, following \citeA{Gomb_08}. No results will be proved under this alternative hypothesis.

\textbf{Test 3}: If
 \begin{align*}
  \sup_{0 \leq t \leq 1} \hcM_n^{(i)}(t)
  - \inf_{0 \leq t \leq 1} \hcM_n^{(i)}(t)
  \geq C_3(\alpha^*) ,
 \end{align*}
 then we conclude that there was a temporary change in parameter \ $\theta^{(i)}$
 along the sequence \ $X_0$, $X_1$, \dots, $X_n$.

Critical values are obtained from the limit distributions in Corollary
 \ref{sup}, namely, from the identities
 \begin{align*}
  \PP\left( \sup_{0 \leq t \leq 1} B(t) \geq x \right)
  &= \ee^{-2 x^2} , \quad x \geq 0,\\
  \PP\left( \sup_{0 \leq t \leq 1} |B(t)| \geq x \right)
  &= 2 \sum_{k = 1}^\infty (-1)^{k+1} \ee^{-2 k^2 x^2} , \quad x \geq 0,\\
  \PP\left( \sup_{0 \leq t \leq 1} B(t) - \inf_{0 \leq t \leq 1} B(t) \geq x \right)
  &= 1 - 2 \sum_{k = 1}^\infty (4 k^2 x^2 -1) \ee^{-2 k^2 x^2} , 
 \end{align*}
 respectively, where \ $(B(t))_{0 \leq t \leq 1}$ \ is a Brownian bridge (for
 the third relationship, see \citeNP{Kuip_60}).

\section{The test under the alternative hypothesis}\label{change-point_estimation}
In this section we will state the two main results of this paper which extend the results of \citeA{HPS_07} and \citeA{Gomb_08} to the INAR($p$) process.
\subsection{Consistency of the test}
The following theorem, the analogue of Theorem 3.1 in \citeA{HPS_07}, describes the behaviour of the maximum of the test process if a change occurs in the mean of the innovation. An immediate consequence of the theorem is that the maximum of the process tends to infinity stochastically as \ $n \to \infty$, \ which suffices for the weak consistency of the proposed test.

\begin{Thm}\label{HA_mu_1}
Suppose that \ $\mathrm{H}_\mathrm{A}$ \ holds with
 \ $\tau = \max(\lfloor n \rho \rfloor,1)$, \ $\rho \in \left(0, 1\right)$, \ and the change is
 only in the innovation mean, namely, the innovation mean changes from
 \ $\mu'$ \ to \ $\mu''$, \ where \ $\mu' > \mu'' > 0$.
\ Suppose that \ $\EE(X_0^6) < \infty$, \dots, $\EE(X_{-p+1}^6) < \infty$,
 \ $\EE(\vare_1^6) < \infty$, \ $\EE(\vare_{\tau+1}^6) < \infty$,
 \ $\alpha_1 + \cdots + \alpha_p < 1$ \ and \ $\mu > 0$ \ hold both before and after the change.
Assume that \ $\alpha_1 + \cdots + \alpha_p > 0$ \ or \ $\sigma^2 > 0$.
\ Then for any \ $\gamma \in \left(0, \frac14\right)$ \ we have
 \[
   \max_{1 \leq k \leq n} \sum_{j = 1}^k \hM_j^{(n)}
   = n \psi
     + \OO_{\PP}(n^{1-\gamma}) \qquad \text{as \ $n \to \infty$,}
 \]
 with
 \[
   \psi := \rho (1-\rho) (\mu'-\mu'')
     \be_{p+1}^\top
     \bC'' \tbQ^{-1} \bC'
     \be_{p+1} > 0 ,
 \]
 where \ $\be_{p+1}$ \ is the \ $p+1$-st \textup{(}$p+1$-dimensional\textup{)}
 unit vector,
 \[
   \bC' := \EE\left( \begin{bmatrix}
                      \tbX' \\
                      1
                     \end{bmatrix}
                     \begin{bmatrix}
                      \tbX' \\
                      1
                     \end{bmatrix}^\top \right) , \qquad
   \bC'' := \EE\left( \begin{bmatrix}
                       \tbX'' \\
                       1
                      \end{bmatrix}
                      \begin{bmatrix}
                       \tbX'' \\
                       1
                      \end{bmatrix}^\top \right) ,
 \]
 where the distributions of
 \[
   \tbX' = \begin{bmatrix}
            \tX'_0 \\
            \vdots \\
            \tX'_{-p+1}
           \end{bmatrix} , \qquad
   \tbX'' = \begin{bmatrix}
             \tX''_0 \\
             \vdots \\
             \tX''_{-p+1}
            \end{bmatrix}
 \]
 are the unique stationary distributions of the Markov chains
 \ $(\bX_k)_{0 \leq k \leq \lfloor n \rho \rfloor}$ \ and
 \ $(\bX_k)_{\lfloor n \rho \rfloor + 1 \leq k \leq n}$, \ respectively, and
 \[
   \tbQ := \rho \bC' + (1-\rho) \bC''.
 \]
\end{Thm}

\begin{remark}
Theorem \ref{HA_mu_1} can be reformulated for a change in the other parameters
 as well. 
In other cases we have to investigate the maximum or minimum of
 \ $\sum_{j=1}^k\hM_j^{(n)}X_{j-q}$ \ for some $1 \leq q \leq p$. 
\ This will be explored in more detail in \ref{adapt}. 
One warning is given here for emphasis, although it will be repeated later on, namely, the one-sided test should only be used if we are certain
 that only one parameter of the process has changed.
\end{remark}

\subsection{Estimation of the change point}

Based on the score vector analogy described in Appendix A, the estimator of
 \ $\tau$ \ is
 \begin{align}\label{chp1}
   \htau_n
   := \min \left\{ k \in \{1, \dots, n\}
                  : \sum_{j = 1}^k \hM_j^{(n)}
                    = \max_{1 \leq m \leq n} \sum_{j = 1}^m \hM_j^{(n)} \right\}
 \end{align}
 for the downward one-sided test,
 \begin{align}\label{eq:chp2}
   \htau_n
   = \min \left\{ k \in \{1, \dots, n\}
                  : \sum_{j = 1}^k \hM_j^{(n)}
                    = \min_{1 \leq m \leq n}
                       \sum_{j = 1}^m \hM_j^{(n)} \right\}
 \end{align}
 for the upward one-sided test, and
 \begin{align}\label{eq:chp3}
   \htau_n
   = \min \left\{ k \in \{1, \dots, n\}
                  : \left| \sum_{j = 1}^k \hM_j^{(n)} \right|
                    = \max_{1 \leq m \leq n} \left| \sum_{j = 1}^m \hM_j^{(n)} \right|
          \right\}
 \end{align}
 for the two-sided test.

If there is a change in the coefficient \ $\alpha_i$, \ then the estimator
 of \ $\tau$ \ is based on the process
 \ $\sum_{j = 1}^k \hM_j^{(n)} X_{j-i}$, \ $k = 1, \dots, n$.

\begin{Thm}\label{HA_mu_2}
Under the assumptions of Theorem \ref{HA_mu_1}, we have
 \[
   \htau_n - \lfloor n \rho \rfloor
   = \OO_{\PP}(1) \qquad \text{as \ $n \to \infty$,}
 \]
 where \ $\htau_n$ \ is defined by \eqref{chp1}.
\end{Thm}

\begin{remark}
We used the change-point estimator from \eqref{chp1} because in Theorem \ref{HA_mu_1}, the mean $\mu$ changes downward. If it changed upward, then we should apply the estimator from \eqref{eq:chp2} because it corresponds to the appropriate one-sided test. The estimator from \eqref{eq:chp3} corresponds to the two-sided test and satisfies the statement of Theorem \ref{HA_mu_2} regardless of the direction of the change. The proof of this is analogous to the proof of Theorem \ref{HA_mu_2} but even more technical, therefore it will be omitted.
\end{remark}

\begin{remark}
Similarly to Theorem \ref{HA_mu_1}, Theorem \ref{HA_mu_2} holds for a
 change in other parameters as well. We also note that the result is slightly stronger than the similar Proposition 3.1 in \citeA{HPS_07}. Similar results are valid for change in a location parameter (see \citeNP{CsHorv_97}), and in these cases the limit distribution is nondegenerate. Therefore we can conjecture that Theorem \ref{HA_mu_2} cannot be improved upon in terms of convergence rate.
\end{remark}

Theorem \ref{HA_mu_1} and Theorem \ref{HA_mu_2} are proved in Appendix \ref{app_B}. If we would like to prove different versions of these theorems (e.g., if we postulate a change in a parameter other than $\mu$), we need to slightly modify the proofs in several points. Any points where these modifications are not straightforward will be indicated in \ref{adapt}.

\section{Illustration}

Now we provide two real data examples of the use of our method. Since our model includes initial values, the series were not investigated in their full length, but the first $p$ values were taken as the initial values $X_{-p+1}, \ldots, X_0$.

Our first example is the dataset of monthly polio cases in the US, as reported by the Centers for Disease Control and Prevention. It is available online at \citeA{Hyndm_URL} and is 166 long. In \citeA{KLee_09} the authors found a significant decreasing trend in this series, while in \citeA{Davis_09} and \shortciteA{Davis_00} the trend was found insignificant. It is widely agreed (see also \citeNP{Silva_05}) that the underlying process is first-order, which is also supported by the partial autocorrelation function. Therefore we treated it as an INAR(1) process and calculated the CLS estimates given by \eqref{CLSests}. They were $\halpha_1 = 0.30646$ and $\hmu = 0.94091$. The maximum of the absolute value of $\hcM_{166}^{(1)}$ was 1.2647 and the maximum of the absolute value of $\hcM_{166}^{(2)}$ was 1.1232. Applying the two-sided test simultaneously to the two parameters and requiring an overall significance level of 0.05, the critical value for each component is 1.48 (the individual significance levels are $1-\sqrt{0.95} \approxeq 0.0253$), therefore, the null hypothesis is not rejected.

Our second example is a dataset of public drunkenness intakes in Minneapolis, also accessible at \citeA{Hyndm_URL}. This dataset is 139 long. After an examination of the partial autocorrelation function a seasonal INAR(12) model seems a rational choice, but with the assumption that only $\alpha_1$ and $\alpha_{12}$ are nonzero (for another similar calculation, see the real data section in \citeNP{BarczAsymp_10}). The estimates are
\[
\begin{bmatrix}
\halpha_1 \\
\halpha_{12} \\
\hmu
\end{bmatrix}
= \left(\sum_{k=1}^n
\begin{bmatrix}
X_{k-1} \\
X_{k-12} \\
1
\end{bmatrix}
\begin{bmatrix}
X_{k-1} \\
X_{k-12} \\
1
\end{bmatrix}^\top\right)^{-1}
\sum_{k=1}^n X_k
\begin{bmatrix}
X_{k-1} \\
X_{k-12} \\
1
\end{bmatrix}
=
\begin{bmatrix}
0.8154 \\
0.1419 \\
9.6944
\end{bmatrix}.
\]

The maxima of the absolute values of the respective components of $\hcM_n$ are 2.0333, 1.3497 and 1.5788. A comparison with the critical value of 1.545 (individual significance of approximately 0.017) results in the rejection of the null hypothesis. Based on $\left(\sum_{i=1}^k \hM_k^{(n)}X_{k-1}\right)_{k=1}^{139}$ our estimate for the change point is 41 (i.e., the 53rd entry in the original series). Repeating the procedure for the series before and after the change, the null hypothesis is accepted for both of them. For the series after the change, the CLS estimate of $\alpha_{12}$ is negative but an inspection of the partial autocorrelation function reveals that this series is more appropriately modeled as an INAR(1) process, for which the parameter estimates are $\halpha_1 = 0.8915$ and $\hmu = 24.8429$ and the null hypothesis is accepted.

\appendix

\section{The process under the null hypothesis}\label{app_A}

In this section we will give some results for the process under the null hypothesis, including a lemma for ergodic convergence rate. The conditional least squares estimators will be calculated and we will give details on the calculation of the test process \eqref{hcMn}.

\subsection{Regression equations}\label{regreq}
The INAR($p$) process is formally analogous to the AR($p$) process. To exploit this analogy we need to state several regression equations for the process. First we write the equivalent of \eqref{GWI} for the vector-valued process $(\bX_k)_{k \in \NN}$.
\begin{equation}\label{eq:stochregression}
\bX_k = \sum_{i=1}^{p}\sum\limits_{j=1}^{X_{k-i}}\boldsymbol{\xi}_{i,k,j}+\boldsymbol{\vare}_k,
\end{equation}
where
\begin{equation*}
\boldsymbol{\xi}_{1,k,j} = 
\begin{bmatrix}
\xi_{1,k,j} \\
1 \\
0 \\
\vdots \\
0
\end{bmatrix}
,\ 
\boldsymbol{\xi}_{2,k,j} = 
\begin{bmatrix}
\xi_{2,k,j} \\
0 \\
1 \\
\vdots \\
0
\end{bmatrix}
,\ \ldots ,\ 
\boldsymbol{\xi}_{p,k,j} = 
\begin{bmatrix}
\xi_{p,k,j} \\
0 \\
0 \\
\vdots \\
0
\end{bmatrix}
,\ 
\bvare_k = 
\begin{bmatrix}
\vare_k \\
0 \\
0 \\
\vdots \\
0
\end{bmatrix}
.
\end{equation*}

This form makes it even more apparent that the INAR($p$) process is a special multitype branching process with immigration. According to standard literature (see, e.g., \citeNP{Qu_70}), if the matrix
\begin{equation}\label{Amtxdefi}
\bA:=
\begin{bmatrix}
\EE(\bxi_{1,1,1}) & \cdots & \EE(\bxi_{p,1,1})
\end{bmatrix}
=
\begin{bmatrix}
\alpha_1 & \alpha_2 & \cdots & \alpha_{p-1} & \alpha_p \\
1 & 0 & \cdots & 0 & 0 \\
\vdots & \vdots & \ddots & \vdots & \vdots \\
0 & 0 & \cdots & 0 & 0 \\
0 & 0 & \cdots & 1 & 0
\end{bmatrix}
\end{equation}
is primitive (i.e., some power of it is elementwise positive), the ergodicity of the process depends only on the spectral radius $\rho(A)$, and the process is ergodic if $\rho(A) < 1$. In \citeA{BarczAsymp_10}, (2.7), it is shown that this is equivalent to the condition that $\alpha_1 + \ldots + \alpha_p < 1$.

Recalling the $M_k$ martingale differences from \eqref{Mdefi} we can write
\begin{equation}\label{vectorregression}
\bX_k = \bA \bX_{k-1} + (\mu + M_k) \be_1,
\end{equation}
where $\be_1$ is the first unit vector.
Based on \eqref{vectorregression} we obtain
\begin{equation}\label{Ax2}
\begin{split}
\bX_k^{\otimes 2} &= (\bA \bX_{k-1})^{\otimes 2} + ((\mu + M_k)\be_1)^{\otimes 2} + (\bA \bX_{k-1})\otimes((\mu + M_k)\be_1)\\
&\quad+ ((\mu + M_k)\be_1)\otimes(\bA \bX_{k-1}) \\
&= \bA^{\otimes 2} \bX_{k-1}^{\otimes 2} + (\mu + M_k)^2\be_1^{\otimes 2} + (\mu + M_k)(\bA\bX_{k-1})\otimes\be_1\\
&\quad+ (\mu + M_k)\be_1\otimes(\bA\bX_{k-1}),
\end{split}
\end{equation}
where $\otimes$ denotes Kronecker product of matrices.

\subsection{Ergodicity}\label{statmoments}
Under the null hypothesis let us denote by $\tbX$ a random vector with the unique stationary distribution of $(\bX_k)_{k\in \NN}$. Because our process is ergodic, we can apply the ergodic theorem. In its well-known form it states that if $\EE(|g(\tbX)|) < \infty$ for some function $g$, then
\begin{equation}\label{ergodic}
\frac{1}{n}\sum_{k=1}^n g(\bX_k) \as \EE(g(\tbX)).
\end{equation}
This is, for example, Theorem 2 in I.15. in \citeA{Chung_60}. However, instead of the convergence of averages, we will frequently require the convergence of expectations, i.e.,
\begin{equation}\label{fergodic}
\EE(\bX_k^{\otimes \beta}) \to \EE(\tbX^{\otimes \beta}), \qquad \beta\in\NN,
\end{equation}
whenever the right hand side is finite. This is Theorem 14.0.1 in \citeA{MeynTw_09}. The result in \eqref{fergodic} also implies convergence of any component of the matrices. Under the alternative hypothesis we will additionally apply
\begin{equation}\label{vergodic}
\sum_{x \in \NN_0^p} | \PP(\bX_n = x) - \PP(\tbX = x) | \to 0.
\end{equation}
 This result can be found in \citeA[Theorem 4.3]{Orey_71} or in
 \citeA[Theorem 13.1.2]{MeynTw_09}.

The convergence rate in \eqref{fergodic} can be estimated by the following lemma.
\begin{Lem}\label{geom}
There is a constant \ $\pi \in (0, 1)$ \ such that
\[
\|\EE(\bX_k) - \EE(\tbX)\|= \OO(\pi^k),\qquad
\|\EE(\bX_k^{\otimes 2}) - \EE(\tbX^{\otimes 2})\|= \OO(\pi^k).
\]
\end{Lem}
\begin{proof}
We use \eqref{vectorregression} to conclude that
\[
\EE(\bX_k) = \bA \EE(\bX_{k-1}) + \mu \be_1.
\]
Taking the limits as \ $k \to \infty$ \ we have
\[
\EE(\tbX) = \bA \EE(\tbX) + \mu \be_1,
\]
hence
\begin{equation}\label{recursion}
\EE(\bX_k) - \EE(\tbX) = \bA( \EE(\bX_{k-1}) - \EE(\tbX) ).
\end{equation}
Similarly, from \eqref{Ax2} and \ $\EE(M_k^2|\cF_{k-1}) = \balpha_2^\top \bX_{k-1} + \sigma^2 $ \ we have
\begin{equation}\label{eq:vektorcsere}
\begin{split}
\EE(\bX_k^{\otimes 2}) - \EE(\tbX^{\otimes 2}) 
&= \bA^{\otimes 2} (\EE(\bX_{k-1}^{\otimes 2}) - \EE(\tbX^{\otimes 2})) + \big(\balpha_2^\top ( \EE(\bX_{k-1}) - \EE(\tbX) ) \big) \be_1^{\otimes 2} \\
&\quad + \mu (\bA ( \EE(\bX_{k-1}) - \EE(\tbX) ) ) \otimes \be_1  + \mu\be_1 \otimes (\bA ( \EE(\bX_{k-1}) - \EE(\tbX) ) ) \\
&= \bA^{\otimes 2} (\EE(\bX_{k-1}^{\otimes 2}) - \EE(\tbX^{\otimes 2})) + \be_1^{\otimes 2} \balpha_2^\top ( \EE(\bX_{k-1}) - \EE(\tbX) )\\
&\quad + \mu (\bA \otimes \be_1) ( \EE(\bX_{k-1}) - \EE(\tbX) ) + \mu(\be_1 \otimes \bA) ( \EE(\bX_{k-1}) - \EE(\tbX) ).
\end{split}
\end{equation}
Here we used the fact that for any $c \in \RR$ and real vector $\bv$ we have $c\bv = \bv c$, where the second multiplication is a proper matrix product. Furthermore, we used the following property of the Kronecker product: for any matrices $A,B,C,D$ we have $(AB)\otimes(CD)=(A\otimes C)(B\otimes D)$, specifically, if $C$ is a column vector, $(AB)\otimes C=(AB)\otimes (C \cdot 1)=(A\otimes C)(B\otimes 1)=(A\otimes C)B$ (this identity can also be used when the first factor consists of a single factor instead of the second).
Hence,
\[
\begin{bmatrix}
\EE(\bX_k) \!-\! \EE(\tbX) \\
\EE(\bX_k^{\otimes 2}) \!-\! \EE(\tbX^{\otimes 2})
\end{bmatrix}
\!=\!
\begin{bmatrix}
\bA & \bzero \\
\be_1^{\otimes 2} \balpha_2^\top \!+\! \mu (\bA \otimes \be_1)
 \!+\! \mu(\be_1 \otimes \bA) & \bA^{\otimes 2}
\end{bmatrix}\!\!
\begin{bmatrix}
\EE(\bX_{k-1}) \!-\! \EE(\tbX) \\
\EE(\bX_{k-1}^{\otimes 2}) \!-\! \EE(\tbX^{\otimes 2})
\end{bmatrix}\!.
\]
Let us denote the multiplicating matrix on the right hand side by \ $\bD$. \ We note that \ $\bD$ \ is block lower triangular and that due to the properties of the Kronecker product, \ $\rho(\bA^{\otimes 2}) = (\rho(\bA))^2< \rho(\bA)$ (here and throughout the paper, $\rho$ denotes the spectral radius of a matrix). \ From these it is clear that \ $\rho(\bD) = \rho(\bA) < 1.$ \ It is well-known that then there exists an induced matrix norm \ $\vnorm{\cdot}_*$ \ for which
 \ $\rho(\bA) < \vnorm{\bA}_* <1 $. \ This, and the equivalence of vector norms suffice for the proof. 
\end{proof}

\begin{remark}
The finiteness of the respective moments of the stationary distribution can be derived using the same approach as in the proof of formulae (2.2.3), (2.2.4) and (2.2.10) in \citeA{BarczAsymp_10}. We note that the stationary distribution has exactly as many finite moments as the innovation distribution and the initial distributions have in common, because the Bernoulli distribution is bounded and therefore all of its moments are finite.
\end{remark}

\subsection{The variance of partial sums of the process}

The following lemma will be used for estimations in the next section and shows that while the values of the process are not independent, their dependence is very weak in the sense that the variance of the partial sums up to time $n$ is only linearly increasing with $n$.

\begin{Lem}\label{vari}
In a stable INAR($p$) model under \ $\mathrm{H_0}$
\begin{enumerate}
\item[\textup{(i)}]  $\var(X_1+X_2+\ldots+X_n) = \sum_{i,j=1}^n \cov(X_i, X_j) =  \OO(n),$
\item[\textup{(ii)}] $\var(X_1X_{1-q}+X_2X_{2-q}+\ldots+X_nX_{n-q}) = \sum_{i,j=1}^n \cov(X_iX_{i-q}, X_jX_{j-q}) = \OO(n)$ \ for all \ $0 \leq q \leq p-1.$
\end{enumerate}
\end{Lem}

\begin{proof}
Although the lemma is stated for the process \ $(X_n)_{n \in \NN}$,\ calculations will require that we investigate the process \ $(\bX_n)_{n \in \NN}$. \ Therefore, we will prove the following statements:
\begin{equation}\label{i}
\vnorm{\var(\bX_1+\bX_2+\ldots+\bX_n)} = \OO(n)
\end{equation}
in the place of (i) and
\begin{equation}\label{ii}
\vnorm{\var(\bX_1^{\otimes 2}+\bX_2^{\otimes 2}+\ldots+\bX_n^{\otimes 2})} = \OO(n)
\end{equation}
in the place of (ii).

First we will prove \eqref{i}.
\eqref{fergodic} implies that \ $\left(\vnorm{\var(\bX_i)}\right)_{i \in \NN}$ is a (convergent and hence) bounded series. Let us denote its upper bound by \ $U_1$. \ This observation means that we will only need to investigate the sum \ $\sum_{i,j=1}^n \cov(\bX_i,\bX_j) - \sum_{i=1}^n \cov(\bX_i,\bX_i)$, \ since the latter sum is clearly \ $\OO(n)$. \

First we will use \eqref{vectorregression}. 
It is immediate from \eqref{vectorregression} that
\[
\bX_k - \EE(\bX_k) = \bA (\bX_{k-1} - \EE(\bX_{k-1})) + M_k\be_1.
\]
Let us now fix \ $1 \leq i < j$ \ and write
\begin{equation}\label{covregr}
\begin{split}
&\cov(\bX_i,\bX_j)=\EE\left[(\bX_i-\EE(\bX_i))(\bX_j-\EE(\bX_j))^\top\right]\\ 
&=\EE\left[(\bX_i-\EE(\bX_i))\EE(\bX_j-\EE(\bX_j)|\cF_{j-1})^\top\right] \\
&=	\EE\left[(\bX_i-\EE(\bX_i))(\bA (\bX_{j-1} - \EE(\bX_{j-1})))^\top\right] = \cov(\bX_i,\bX_{j-1})\bA^\top.
\end{split}
\end{equation}
If we perform the calculations for the case \ $1 \leq j < i$ as well, \ we can see that, after \ $|j-i|$ \ iterations,
\begin{equation}\label{covariA}
\cov(\bX_i,\bX_j) = \bA^{(i-j)_+}\var(\bX_{\min(i,j)})(\bA^\top)^{(j-i)_+}.
\end{equation}
Because \ $\rho(\bA) < 1$, \ there exists a matrix norm \ $\vnorm{\cdot}_*$ \ for which $\rho(\bA) <\vnorm{\bA}_*=:\pi < 1$. \ With this norm, \eqref{covariA}, and the boundedness of \ $(\var(\bX_{i}))_{i \in \NN}$ we can establish
\[
\vnorm{\cov(\bX_i,\bX_j)}_* = \OO(\pi^{|i-j|}),
\]
which yields \eqref{i} immediately.

For \eqref{ii} our reasoning will be very similar, although with more tedious calculations. First we note that \eqref{fergodic} implies boundedness for $\vnorm{\var(\bX_i^{\otimes 2})}$ also. From \eqref{Ax2} we have
\begin{equation*}
\begin{split}
\EE(\bX_k^{\otimes 2} - \EE(\bX_k^{\otimes 2})|\cF_{k-1}) &=
	\bA^{\otimes 2}(\bX_{k-1}^{\otimes 2} - \EE(\bX_{k-1}^{\otimes 2})) 
		+ \balpha_2^\top (\bX_{k-1}-\EE(\bX_{k-1}))\be_1^{\otimes 2} \\
	& \quad + \mu[\bA(\bX_{k-1} \!-\! \EE(\bX_{k-1}))]\otimes\be_1 
		+ \mu\be_1\otimes[\bA(\bX_{k-1} \!-\! \EE(\bX_{k-1}))] \\
	&= \bA^{\otimes 2}(\bX_{k-1}^{\otimes 2} - \EE(\bX_{k-1}^{\otimes 2})) 
			+ \be_1^{\otimes 2} \balpha_2^\top (\bX_{k-1}-\EE(\bX_{k-1})) \\
		& \quad + \mu[\bA \otimes \be_1 ](\bX_{k-1} \!-\! \EE(\bX_{k-1}))\otimes\ 
			+ \mu(\be_1\otimes\bA)(\bX_{k-1} \!-\! \EE(\bX_{k-1}))
	,
\end{split}
\end{equation*}
analogously to \eqref{eq:vektorcsere}.
Now, similarly to \eqref{covregr} we get, for \ $1 \leq i < j$, \
\begin{equation*}
\begin{split}
\cov(\bX_{i}^{\otimes 2}, \bX_j^{\otimes 2}) &= \cov(\bX_{i}^{\otimes 2}, \bX_{j-1}^{\otimes 2}) (\bA^{\otimes 2})^\top + \cov(\bX_{i}^{\otimes 2}, \bX_{j-1}) \balpha_2 (\be_1^{\otimes 2})^{\top} \\
& \quad + \mu \cov(\bX_i^{\otimes 2}, \bX_{j-1}) \left(\bA \otimes \be_1\right)^\top + \mu \cov(\bX_i^{\otimes 2}, \bX_{j-1}) \left( \be_1\otimes \bA \right)^\top.
\end{split}
\end{equation*}
Here 
\[
\cov(\bX_i^{\otimes 2}, \bX_{j-1}) := \EE[(\bX_i^{\otimes 2} - \EE(\bX_i^{\otimes 2}))(\bX_{j-1} - \EE(\bX_{j-1}))^\top],
\]
a \ $p^2 \times p$ \ matrix.
Also similarly to \eqref{covregr} we have
\[
\cov(\bX_{i}^{\otimes 2}, \bX_j) = \cov(\bX_i^{\otimes 2}, \bX_{j-1})\bA^\top.
\]
Summarizing, we get the following regression:
\begin{equation}
	\begin{bmatrix}
		\cov(\bX_{i}^{\otimes 2}, \bX_{j})^\top \\
		\cov(\bX_{i}^{\otimes 2}, \bX_{j}^{\otimes 2})^\top
	\end{bmatrix}
=
	\begin{bmatrix}
		\bA^{\otimes 2} & \bzero \\
		\be_1^{\otimes 2}\balpha_2^\top + \mu(\bA \otimes \be_1) + \mu(\be_1 \otimes \bA) & \bA
	\end{bmatrix}
	\begin{bmatrix}
		\cov(\bX_{i}^{\otimes 2}, \bX_{j-1})^\top \\
		\cov(\bX_{i}^{\otimes 2}, \bX_{j-1}^{\otimes 2})^\top
	\end{bmatrix}
.
\end{equation}
Note that the multiplicating matrix on the right hand side is just \ $\bD$ \ from the proof of Lemma \ref{geom}. \ Now, similarly to \eqref{covariA}, we have
\begin{equation}
	\begin{bmatrix}
		\cov(\bX_{i}^{\otimes 2}, \bX_{j})^\top \\
		\cov(\bX_{i}^{\otimes 2}, \bX_{j}^{\otimes 2})^\top
	\end{bmatrix}
=
	\bD^{(j-i)_+}
	\begin{bmatrix}
		\cov(\bX_{\min(i,j)}^{\otimes 2}, \bX_{\min(i,j)})^\top \\
		\cov(\bX_{\min(i,j)}^{\otimes 2}, \bX_{\min(i,j)}^{\otimes 2})^\top
	\end{bmatrix}
	\left(\bD^\top\right)^{(i-j)_+}.
\end{equation}
Now we only need to note that \ $(\cov(\bX_i^{\otimes 2}, \bX_{i}))_{i \in \NN}$ \ is a bounded sequence (due to \eqref{fergodic}), and we can finish the proof of \eqref{ii} in the same way as \eqref{i}.
\end{proof}

\subsection{Strong consistency of the estimates}
In this subsection we want to show with our notations the fact, proven already by \citeA{DuLi_91} for $\btheta$, that the estimates for the coefficients and the mean and variance of the innovation given in Section \ref{sec:parest} are strongly consistent, i.e.,
\begin{equation}
\hbtheta^{(n)} \as \btheta \quad \text{and} \quad \widehat{\sigma}^2_n \as \sigma^2.
\end{equation}

First of all, it is a matter of simple calculations (see \ref{cmoments}) and a straightforward application of \eqref{ergodic} that, provided that the second moment of $\vare_1$ exists,
\begin{equation}\label{eq:M2limit}
\frac{1}{n}\sum_{k=1}^n M_k^2 \to \balpha_2 \EE(\tbX) + \sigma^2,
\end{equation}
with $\balpha_2$ given in \eqref{eq:balpha2defi}.
Hence, taking the limits of the expectations in \eqref{vectorregression} and \eqref{Ax2} we have
\begin{equation*}
\EE(\tbX)=\bA\EE(\tbX)+\mu  \quad \text{and} \quad \EE(\tbX^{\otimes 2})=\bA^{\otimes 2}\EE(\tbX^{\otimes 2}) + (\mu^2 + \balpha_2^{\top} \EE(\tbX) + \sigma^2 ) + \mu (\be_1 \otimes \bA \EE(\tbX) + \bA \EE(\tbX) \otimes \be_1).
\end{equation*}

From this we can deduce, using \eqref{fergodic} and an argument from the proof of Lemma \ref{psiasymp},
\begin{align*}
\hbtheta^{(n)} &=
\left(\frac{\bQ_n}{n}\right)^{-1}
\frac{1}{n}
\sum_{k=1}^n
      X_k
      \begin{bmatrix}
       \bX_{k-1} \\
       1
      \end{bmatrix}
\\
&\to
\EE
\left(\begin{bmatrix}
\tbX \\
1
\end{bmatrix}
\begin{bmatrix}
\tbX \\
1
\end{bmatrix}
^\top
\right)^{-1}
\EE
\left(
(\balpha^\top \tbX + \mu)
\begin{bmatrix}
\tbX \\
1
\end{bmatrix}
\right)
=
\EE
\left(\begin{bmatrix}
\tbX \\
1
\end{bmatrix}
\begin{bmatrix}
\tbX \\
1
\end{bmatrix}
^\top
\right)^{-1}
\EE
\left(\begin{bmatrix}
\tbX \\
1
\end{bmatrix}
\begin{bmatrix}
\tbX \\
1
\end{bmatrix}
^\top
\right)
\btheta = \btheta.
\end{align*}

A similar result can be derived for the estimate of $\sigma^2$. By recalling \eqref{eq:M2limit} and computing the strong limit of the other summands in \eqref{eq:barsigma}, we obtain the strong consistency of $\overline{\sigma}^2_n$ immediately. The same reasoning shows that if the second moment of the stationary distribution is finite (in this case we already know that $\hbtheta^{(n)}$ is a consistent estimator), then the limits of the estimators $\overline{\sigma}^2_n$ and $\widehat{\sigma}^2_n$ are the same almost surely; hence, the strong consistency of $\widehat{\sigma}^2_n$ is established.

\subsection{Proof of Theorem \ref{theta}}

We recall that the theorem was stated under the null hypothesis, so we will assume it to hold. Let

 \begin{align*}
  \bcZ_n(t) := \frac{1}{\sqrt{n}}\sum_{k=1}^\nt \bZ_k ,
    \quad t \in [0,1] , \qquad
    \bZ_k := M_k
           \begin{bmatrix}
            \bX_{k-1} \\
            1 
           \end{bmatrix} , \qquad k = 1, 2, \dots
 \end{align*}

The proof will be based on the following theorem.
 
\begin{Thm}\label{MCLT}
Under the assumptions of Theorem \ref{theta}
 \begin{align*}
  \bcZ_n \distr \bI^{1/2} \, \bcW, \qquad n \to \infty,
 \end{align*}
 where \ $(\bcW(t))_{0 \leq t \leq 1}$ \ is a $(p+1)$-dimensional standard Wiener
 process and
 \begin{equation}\label{Imtxdefi}
 \bI:=\EE\left(
 (\balpha^\top_2 \tbX + \sigma^2)
 \mx{\tbX}\mx{\tbX}^\top
 \right).
 \end{equation}
 
\end{Thm}

\begin{proof}
By \eqref{ergodic} we have $n^{-1} \bI_n \as \bI$, and since \ $\widehat{\sigma}^{2}_n$ \ and \ $\hbtheta_n$ \ are strongly consistent estimators, therefore we have \ $ n^{-1}\hbI_n \as \bI$ as well.
We will use the martingale central limit theorem for the martingale differences \ $\frac{1}{\sqrt{n}} \bZ_k, \ n \in \NN, \ k=1,\,2,\,\ldots,\,n$. To compute the variance function, we write
\begin{align*}
\frac{1}{n}\sum_{k=1}^\nt\EE(\bZ_k \bZ_k^\top | \cF_{k-1}) &= \frac{\nt}{n} \frac{1}{\nt}\sum_{k=1}^\nt \EE(M_k^2|\cF_{k-1})
\begin{bmatrix}
\bX_{k-1}\\
1
\end{bmatrix}
\begin{bmatrix}
\bX_{k-1}\\
1
\end{bmatrix}
^\top
\\
&
=
\frac{\nt}{n} \frac{1}{\nt}\sum_{k=1}^\nt (\balpha_2^\top \bX_{k-1}+\sigma^2)
\begin{bmatrix}
\bX_{k-1}\\
1
\end{bmatrix}
\begin{bmatrix}
\bX_{k-1}\\
1
\end{bmatrix}
^\top
\\
& \as
t \EE \left(
(\balpha_2^\top \tbX+\sigma^2)
\begin{bmatrix}
\tbX\\
1
\end{bmatrix}
\begin{bmatrix}
\tbX\\
1
\end{bmatrix}
^\top
\right)
= t \bI .
\end{align*}
It remains to check the so-called conditional Lindeberg condition:
\begin{align*}
\sum_{k=1}^\nt
\EE
\left(\left.
\vnorm{\frac{1}{\sqrt{n}}\bZ_k}^2 \chi _{\left\{\vnorm{n^{-1/2}\bZ_k} > \delta\right\} } \right\rvert \cF_{k-1}
\right)
&\leq
\frac{1}{\delta^2}
\sum_{k=1}^\nt
\EE
\left(\left.
\vnorm{\frac{1}{\sqrt{n}} \bZ_k}^4
\right\rvert
\cF_{k-1}
\right)
\\
&=
\frac{1}{\delta^2n^2}
\sum_{k=1}^{\nt}
\EE\!
\left(
M_k^4|\cF_{k-1}
\right)\!\!
\left(X_{k-1}^2 \!+\! \ldots \!+\! X_{k-p}^2\!+\!1\right)^2
\\
&=
\frac{1}{\delta^2n^2}
\sum_{k=1}^{\nt}
P(\bX_{k-1}),
\end{align*}
where \ $P$ \ is a polynomial of degree six, because \ $\EE(M_k^4|\cF_{k-1})$ \ is a second-degree polynomial of $\bX_{k-1}$ (this is is detailed in \ref{cmoments}). The sixth moment of the stationary distribution is finite due to the assumptions, hence \eqref{ergodic} implies
\[
\frac{1}{\nt}
\sum_{k=1}^{\nt}
P(\bX_{k-1})
\as
\EE(P(\tbX)) < \infty .
\]
This means
\[
\frac{1}{\delta^2n^2}
\sum_{k=1}^{\nt}
P(\bX_{k-1})
\as 0 ,
\]
implying Lindeberg's condition. All the conditions of the martingale central limit theorem have been checked; the proof is therefore complete.
\end{proof}

Having proved Theorem \ref{MCLT}, we now turn to the proof of Theorem \ref{theta}. For this, let us introduce the notation
\begin{align*}
  \hbZ_k^{(n)} := \hM_k^{(n)}
           \begin{bmatrix}
            \bX_{k-1} \\
            1 
           \end{bmatrix} , \qquad k = 1, 2, \dots
 \end{align*}
First we note that
\[
\sum_{k=1}^\nt \hbZ_k^{(n)} = \sum_{k=1}^\nt \bZ_k + \sum_{k=1}^\nt (\hbZ_k^{(n)} - \bZ_k) = \sum_{k=1}^\nt \bZ_k + \sum_{k=1}^\nt (\hM_k^{(n)} - M_k)
\begin{bmatrix}
\bX_{k-1} \\
1
\end{bmatrix}.
\]
Recalling the definitions of $M_k$ and $\hM_k^{(n)}$,
\begin{align*}
\hM_k^{(n)} - M_k &= \left(X_k - \hbtheta_n^\top
\begin{bmatrix}
\bX_{k-1} \\
1
\end{bmatrix}
\right)
-
\left(
X_k - \btheta^\top
\begin{bmatrix}
\bX_{k-1} \\
1
\end{bmatrix}
\right) = \left(
\btheta - \hbtheta_n \right)^\top
\begin{bmatrix}
\bX_{k-1} \\
1
\end{bmatrix}.
\end{align*}
We have
\begin{equation}\label{CLSerror}
\begin{split}
\hbtheta_n - \btheta &=
	\bQ_n^{-1}
	\left(
		\sum_{k=1}^n
		X_k
		\begin{bmatrix}
			\bX_{k-1} \\
			1
		\end{bmatrix}
	\right)
	- \btheta	\\
&=
	\bQ_n^{-1}
	\left(
		\sum_{k=1}^n
		X_k
		\begin{bmatrix}
			\bX_{k-1} \\
			1
		\end{bmatrix}
		-
		\sum_{k=1}^n
		\begin{bmatrix}
			\bX_{k-1} \\
			1
		\end{bmatrix}
		\begin{bmatrix}
			\bX_{k-1} \\
			1
		\end{bmatrix}
		^\top
		\btheta^\top
	\right) \\
&=
	\bQ_n^{-1}
	\left(
		\sum_{k=1}^n
		M_k
		\begin{bmatrix}
			\bX_{k-1} \\
			1
		\end{bmatrix}
	\right),
\end{split}
\end{equation}
hence
\begin{equation}\label{CLSapprox}
\hbI_n^{-1/2} \sum_{k=1}^\nt \hbZ_k^{(n)} = 
\sqrt{n}\,\hbI_n^{-1/2} \left(
	\sum_{k=1}^\nt \frac{1}{\sqrt{n}}\bZ_k - \bQ_\nt \bQ_n^{-1} \sum_{k=1}^n \frac{1}{\sqrt{n}} \bZ_k
\right).
\end{equation}
In the next step we notice that according to \eqref{ergodic} we have
\begin{align*}
\bQ_\nt \bQ_n^{-1} = \frac{\nt}{n}\left( \frac{1}{\nt}\bQ_\nt\right) \left(
\frac{1}{n}\bQ_n
\right)^{-1} \as t \tbQ_{\mathrm{H_0}} \tbQ_{\mathrm{H_0}}^{-1} = t \bE_{p+1} \qquad \forall t \in [0,1],
\end{align*}
where $\bE_{p+1}$ is the $p+1$-dimensional identity matrix and
\[
\tbQ_{\mathrm{H_0}}:= \EE\left(\mx{\tbX}\mx{\tbX}^\top\right).
\]
Now we apply \eqref{CLSapprox}, Theorem \ref{MCLT}, and 
\[
\sqrt{n}\,\hbI_n^{-1/2} \as \bI^{-1/2} ,
\]
to conclude that
\[
\left(\hbI_n^{-1/2} \sum_{k=1}^\nt \hbZ_k^{(n)}\right)_{t \in [0,1]} \distr (\bcW(t) - t \bcW(1))_{t \in [0,1]}. \qed
\]

\section{The process under the alternative hypothesis}\label{app_B}
While in Theorem \ref{theta} we were able to consider longer and longer samples taken from the same process, this approach has to be modified for the alternative hypothesis. More precisely, we have to consider a series of time-inhomogeneous INAR($p$) processes, where the $n$-th one has a point of change at $\nrho$ (we will suppress this in the notation for simplicity). Now, the parts of these processes before the change (i.e., $(X_i)_{i=1}^{\nrho}$) can be handled as a sample taken from an infinite INAR($p$) process (at least in distribution), but this is not true for the second part (i.e., $(X_i)_{i \geq \nrho + 1}$), because the initial distribution of this process depends on $n$.

Therefore, for a rigorous analysis we need to refine the results of \ref{statmoments}.

\subsection{Ergodicity under the alternative hypothesis}
We have
\begin{equation}\label{aergodic}
\frac{1}{n - \nrho}\sum_{k=\nrho+1}^n g(\bX_k) \stoch \EE(g(\tbX'')),
\end{equation}
where \ $g:\NN_0^p \to \RR$ \ with \ $\EE(|g(\tbX)|) < \infty$. 
\ Indeed, for an arbitrary \ $\vare>0$ 
\begin{align*}
 &\PP\left( \left| \frac{1}{n - \nrho}\sum_{k=\nrho+1}^n g(\bX_k)
                   - \EE(g(\tbX'')) \right| > \vare \right) \\
 &= \sum_{x \in \NN_0^p}
     \PP\left( \left| \frac{1}{n - \nrho}\sum_{k=\nrho+1}^n g(\bX_k)
                      - \EE(g(\tbX''))\right| > \vare
               \, \Bigg| \, \bX_{\nrho}=x \right)
     \PP(\bX_{\nrho}=x) ,
\end{align*}
and
\[
  \PP\left( \left| \frac{1}{n - \nrho}\sum_{k=\nrho+1}^n g(\bX_k)
                   - \EE(g(\tbX''))\right| > \vare
            \, \Bigg| \, \bX_{\nrho}=x \right) \to 0
\]
by the ergodic theorem for each \ $x \in \NN_0^p$, \ additionally 
\[
\PP(\bX_{\nrho}=x) \leq |\PP(\bX_{\nrho}=x)-\PP(\tbX = x)| + \PP(\tbX = x),
\]
and one can use \eqref{vergodic}. We will also apply that for all \ $\vare > 0$ \ there exists \ $\nu$ \ such
 that
\begin{equation}\label{afergodic}
 \|\EE(\bX_{\nrho + k}) - \EE(\tbX'')\| < \vare \qquad
 \text{for all \ $n \geq \nu$ \ and all \ $k \geq \nu$.}
\end{equation}
First observe that there exists \ $\pi'' \in (0,1)$ \ such that
 \ $\|\EE(\bX_{\nrho+k}) - \EE(\tbX'')\|
    \leq (\pi'')^k \|\EE(\bX_{\nrho}) - \EE(\tbX'')\|$
 \ for all \ $k \in \NN$, \ see \eqref{recursion}.
Next, for all \ $\eta > 0$, \ choose \ $\nu_1$ \ and \ $\nu_2$ \ such that
 \ $(\pi'')^k < \eta$ \ for all \ $k \geq \nu_1$ \ and
 \ $\|\EE(\bX_{\nrho}) - \EE(\tbX')\| < \eta$ \ for all \ $n \geq \nu_2$.
\ Hence
\[
  \|\EE(\bX_{\nrho + k}) - \EE(\tbX'')\|
  \leq \eta \bigl( \|\EE(\bX_{\nrho}) - \EE(\tbX')\|
                   + \| \EE(\tbX') - \EE(\tbX'')\| \bigr)
  \leq \eta^2 + \eta \| \EE(\tbX') - \EE(\tbX'')\| .
\]

\subsection{Behaviour of the estimates under the alternative hypothesis}\label{altests}

In this subsection we will investigate the weak limit of the CLS estimates under the conditions of Theorem \ref{HA_mu_1}. We will also show how the quantity \ $\psi$ \ arises naturally if we would like to know how much the misestimation of the parameters influences our test process.
\begin{Lem}\label{psiasymp}
Under the assumptions of Theorem \ref{HA_mu_1} we have
\[\hbtheta_n \stoch \tbtheta := \begin{bmatrix} \tbalpha \\ \tmu \end{bmatrix} := \tbQ^{-1}
               \left( \rho
                      \bC'
       \begin{bmatrix}
        \balpha \\   
        \mu'
       \end{bmatrix}
                     + (1-\rho)
                        \bC''
       \begin{bmatrix}
        \balpha \\   
        \mu''
       \end{bmatrix} \right).
       \]
\end{Lem}
\begin{proof}
By \eqref{ergodic} and \eqref{aergodic} we obtain
 \begin{align*}
  \frac{1}{n} \bQ_{n}
   = \frac{1}{n}
      \bQ_{\nrho}
      + \frac{1}{n}
        \sum_{k = \lfloor n \rho \rfloor + 1}^n
         \begin{bmatrix}
          \bX_{k-1} \\
          1
         \end{bmatrix}
         \begin{bmatrix}
          \bX_{k-1} \\
          1
         \end{bmatrix}^\top
   \stoch \rho \bC'
        + (1-\rho) \bC''
    = \tbQ,
 \end{align*}
 as \ $n \to \infty$. Moreover, we can notice that for a homogeneous model the process
 \[
  \bU_k:=
  \begin{bmatrix}
  X_k \\
  \bX_{k-1}
  \end{bmatrix}, \quad k \in \NN \]
  satisfies a similar recursion to \eqref{eq:stochregression}.
  The equivalent of the matrix $\bA$ can then be shown to have a spectral radius smaller than 1, and by the same reasoning as in \ref{regreq} we can conclude that $(\bU_k)_{k \in \NN}$ is ergodic, and we can apply \eqref{ergodic}.
  Moreover, it is clear that if $\tbU$ denotes a vector with the unique stationary distribution of $(\bU_k)_{k \in \NN}$ then
  \[
	\begin{bmatrix}
	\tU^{(2)} \\
	\tU^{(3)} \\
	\vdots \\
	\tU^{(p+1)}
	\end{bmatrix}
	\distre \tbX
  \]
  and for the components of $\tbU$ we also have
  \[
	\tU^{(1)} \distre \sum_{j=1}^{\tU^{(2)}} \xi_{1, j} + \cdots
	        + \sum_{j=1}^{\tU^{(p+1)}} \xi_{p, j} + \vare,
  \]
  where $\xi_{i,j} \distre \xi_{i,1,1}, i=1, \ldots, p, \ j \in \NN$ and $\vare \distre \vare_1$ such that all these variables are totally independent and also independent of $(\tU^{(2)}, \ldots, \tU^{(p+1)})^\top$. Hence \eqref{ergodic} implies.
 \begin{align*}
  \frac{1}{n}
  \sum_{k=1}^{\lfloor n \rho \rfloor}
   X_k
   \begin{bmatrix}
    \bX_{k-1} \\
    1
   \end{bmatrix}
  &\stoch \rho \EE\left( \left( \sum_{j=1}^{\tX'_0} \xi_{1,j} + \cdots
                             + \sum_{j=1}^{\tX'_{-p+1}} \xi_{p,j} + \vare' \right)
                      \begin{bmatrix}
                       \tbX' \\
                       1
                      \end{bmatrix} \right) \\
  &= \rho \EE\left( \left( \alpha_1 \tX'_0 + \cdots + \alpha_p \tX'_{-p+1}
                           + \mu' \right)
                    \begin{bmatrix}
                     \tbX' \\
                     1
                    \end{bmatrix} \right) 
   = \rho \bC'
     \begin{bmatrix}
      \balpha \\   
      \mu'
     \end{bmatrix}
 \end{align*}
 as \ $n \to \infty$ (here $\vare' \distre \vare_1$).
\ In a similar way, using \eqref{aergodic}
 \[
   \frac{1}{n}
   \sum_{k = \lfloor n \rho \rfloor + 1}^n
    X_k
    \begin{bmatrix}
     \bX_{k-1} \\
     1
    \end{bmatrix}
   \stoch (1-\rho) \bC''
       \begin{bmatrix}
        \balpha \\   
        \mu''
       \end{bmatrix}. \qedhere
 \]
\end{proof}
 
 For the rate of convergence we have the following result:
 
 \begin{Lem}\label{altestconv}
 Under the conditions of Theorem \ref{HA_mu_1} we have
 \[
 \hbtheta_n - \tbtheta = \OO_{\PP}(n^{-1/2}).
 \]
 \end{Lem}
 
 \begin{proof}
 The difference can be decomposed in the following way:
 \begin{equation}\label{diffdecomp}
 \begin{split}
 n^{1/2}(\hbtheta_n - \tbtheta) 
 	= (n^{-1}\bQ_n)^{-1}
 	&n^{-1/2}\left[
 		\sum_{k=1}^{\nrho}X_k \mx{\bX_{k-1}} - \bQ_n\tbQ^{-1}
 		\left(
 			\rho C'
 				\begin{bmatrix}
 				\balpha' \\
 				\mu'
 				\end{bmatrix}
 		\right)
 	\right. \\
 	&\quad \left. +
 		\sum_{k=\nrho+1}^{n}X_k \mx{\bX_{k-1}} - \bQ_n\tbQ^{-1}
 		\left(
 			(1-\rho) C''
 				\begin{bmatrix}
 				\balpha'' \\
 				\mu''
 				\end{bmatrix}
 		\right)
 	\right].
 \end{split}
 \end{equation}
 The first factor converges to \ $\tbQ^{-1}$ \ stochastically, and will therefore be omitted from further calculations. The second factor has been split in two and only the first part will be analyzed in detail. The analysis of the second part is completely analogous. We split the first part in the second factor in the following way:
 \begin{align*}
 	n^{-1/2}&\left[\sum_{k=1}^{\nrho}X_k \mx{\bX_{k-1}} - \bQ_n\tbQ^{-1}
 		\left(
 			\rho \bC'
 				\begin{bmatrix}
 				\balpha' \\
 				\mu'
 				\end{bmatrix}
 		\right) \right] \\
 	&= n^{-1/2} \left(
 		\sum_{k=1}^{\nrho} M_k' \mx{\bX_{k-1}}
 		\right) 
 	+ n^{-1/2}\left(
 		\bQ_{\nrho} 
 		- \rho \bQ_n \tbQ^{-1} \bC'
 		\right)
 		\begin{bmatrix}
 			\balpha \\
 			\mu'
 		\end{bmatrix}.
 \end{align*}
 The first term is
 \[
 n^{-1/2} \left(
 		\sum_{k=1}^{\nrho} \bZ'_k
 		\right),
 \]
 which is asymptotically normal, and therefore \ $\OO_{\PP}(1)$ \ according to Lemma \ref{MCLT} (the same reasoning applies after the change, since Lindeberg's theorem is valid for triangular arrays as well).  We will decompose \ $n^{-1/2}\left(\bQ_{\nrho} - \rho \bQ_n \tbQ^{-1} \bC' \right)$ \  in the following way:
 \begin{align*}
 	n^{-1/2}\left(\bQ_{\nrho} - \rho \bQ_n \tbQ^{-1} \bC' \right) 
 	&= n^{-1/2}\left( \bQ_{\nrho} - \EE(\bQ_{\nrho}) \right) 
 		+ n^{-1/2} \left[ \EE(\bQ_{\nrho}) - \nrho \bC'\right] \\
 	& \quad - n^{-1/2}\{\rho[\bQ_n - \EE(\bQ_n)]\tbQ^{-1}\bC'\} -
 	  n^{-1/2}\{\rho[\EE(\bQ_n) - n\tbQ]\tbQ^{-1}\bC'\}\\ 
 	&\quad - n^{-1/2}\{\rho n \bC' - \nrho \bC' \}.
 \end{align*}
 The last term is deterministic and $\oo(1)$. We know from \eqref{ii} that the variances of the first and third terms are bounded. Denoting the common upper bound by $K$ we have, from Markov's inequality, for all $n$,
 \[
 \PP\left(n^{-1}\vnorm{\bQ_{\nrho} - \EE(\bQ_{\nrho})}^2 > a \right) < \frac{K}{a} \to 0 \text{ as } a \to \infty,
 \]
 and similarly for the third term. Consequently, the first and third terms are \ $\OO_{\PP}(1)$. Recalling Lemma \ref{geom} we have
 \begin{align*}
 \vnorm{\EE(\bQ_{\nrho}) - \nrho \bC'} &= \vnorm{\sum_{k=1}^{\nrho}\left(\EE \mxk \mxk^{\top} - \bC'\right)} \\
 & \leq \sum_{k=1}^{\nrho}\vnorm{\EE \mxk \mxk^{\top} - \bC'} \leq \sum_{k=1}^{\nrho} \pi^k = \OO(1),
 \end{align*}
 because the matrices within the sum consist entirely of the entries of $\bX_k^{\otimes 2} - \tbX^{\otimes 2}$. A similar calculation is valid for the fourth term. This implies the boundedness of the second and fourth terms.
 \end{proof}

\subsection{Proof of Theorem \ref{HA_mu_1}}

Now we turn to the proof of our main results: Theorem \ref{HA_mu_1} and Theorem \ref{HA_mu_2}. We will use the following notations:
\[
M'_k: = X_k - \balpha^\top \bX_{k-1} - \mu', \qquad \bZ'_k: = M_k' \mx{\bX_{k-1}},
\]
and similarly for \ $M''_k$ \ and \ $\bZ''_k$. \
The proof will be given for the process before \ $\nrho$ in detail. \ The analysis of the process after \ $\nrho$ \ can be handled analogously. In the proof we will rely repeatedly on ideas from \citeA{HPS_07}.

The first step is the following decomposition for $k < \nrho$:
\begin{equation}\label{decomp}
\begin{split}
\hM_k^{(n)} & = M_k' + \left[(\balpha - \tbalpha)^\top \EE(\bX_{k-1}) + \left(\mu' - \tmu\right)\right] \\
& \quad + (\balpha - \hbalpha_n)^\top (\bX_{k-1} - \EE(\bX_{k-1})) \\
& \quad + \left[(\tbalpha - \hbalpha_n)^\top \EE(\bX_{k-1}) + \left(\tmu - \hmu_n\right)\right].
\end{split}
\end{equation}
For \ $k \geq \nrho$, the quantities \ $M_k'$ \ and \ $\mu'$ \ have to be replaced with \ $M_k''$ \ and \ $\mu''$, respectively.
Based on \eqref{decomp} we have:
\begin{equation}\label{eq:max_ineq}
\begin{split}
  \left|\max_{1 \leq k \leq n} \left( \sum_{i=1}^k \hM_i^{(n)} - n\psi \right)\right|
  &\leq \max_{1 \leq k \leq n} \left| \sum_{i=1}^{\nrho \wedge k} M_i' + \sum_{i=\nrho+1}^{\nrho\vee k} M_i'' \right| \\
  &\quad + \left|\max_{1 \leq k \leq n} \left[ \sum_{i=1}^{\nrho \wedge k} \left[(\balpha - \tbalpha)^\top \EE (\bX_{i-1}) + (\mu'-\tmu)\right] \right.\right. \\
  &\qquad \left.\left. + \sum_{i=\nrho+1}^{\nrho\vee k} \left[(\balpha - \tbalpha)^\top \EE (\bX_{i-1}) + (\mu''-\tmu)  -n \psi \right]\right]\right| \\
  &\quad + \max_{1 \leq k \leq n} \left| \sum_{i=1}^k (\balpha - \hbalpha_{n})^\top \left( \bX_{i-1} - \EE(\bX_{i-1}) \right)  \right| \\
  &\quad + \max_{1 \leq k \leq n} \left| \sum_{i=1}^k \left[ (\tbalpha - \hbalpha_{n} )^\top \EE(\bX_{k-1})  + (\tmu - \hmu_n )\right] \right|.
\end{split}
\end{equation}
The second term in \eqref{eq:max_ineq} is deterministic. For its more detailed analysis we need the following lemma.

\begin{Lem}\label{lem:asymperror}
Under the assumptions of Theorem \ref{HA_mu_1} and using the notations from there we have
\[ (\btheta' - \tbtheta)^\top
\begin{bmatrix}
\EE (\tbX') \\
1
\end{bmatrix}
= \frac{\psi}{\rho} > 0, \qquad (\btheta'' - \tbtheta)^\top \begin{bmatrix}
\EE (\tbX'') \\
1
\end{bmatrix}
= -\frac{\psi}{1-\rho} < 0 \]
with
\[
\btheta' := \begin{bmatrix}
\balpha \\
\mu'
\end{bmatrix},
\qquad
\btheta'' := \begin{bmatrix}
\balpha \\
\mu''
\end{bmatrix}
.
\]
\end{Lem}

\begin{proof}
We consider
 \begin{equation*}
  \begin{bmatrix}
   \balpha - \tbalpha \\
   \mu' - \tmu
  \end{bmatrix}
  = (1-\rho) \tbQ^{-1}  
               \bC''
    \begin{bmatrix}
     \bzero \\
     \mu' - \mu''
    \end{bmatrix}.
 \end{equation*}
We can write
 \[
   \begin{bmatrix}
    \EE(\tbX') \\
    1
   \end{bmatrix}
   = \bC'
    \begin{bmatrix}
     \bzero \\
     1
    \end{bmatrix},
 \]
 whence the first equality in the statement is immediate. The second equality can be proven by the same reasoning. Now we need to show that $\psi>0.$ Indeed,
 \[
 \frac{\psi}{\rho} = (1-\rho) (\mu'-\mu'')
      \be_{p+1}^\top
      \bC' \tbQ^{-1} \bC''
      \be_{p+1},
 \]
 and \ $\bC', \ \bC''$ \ and \ $\tbQ^{-1}$ \ are positive definite matrices. The first two are covariance matrices and because at least one of the offspring and innovation distributions is nondegenerate, they will be positive definite (see \ref{sec:invert}). The matrix $\tbQ^{-1}$ is then the inverse of a convex combination of positive definite matrices, hence positive definite itself.
\end{proof}

Now the second term in \eqref{eq:max_ineq} can be rewritten as
\begin{align*}
&\left|\max_{1 \leq k \leq n} \left[ (\btheta' - \tbtheta)^\top \sum_{i=1}^{\nrho \wedge k} \left[\left(  \EE \mx{\bX_{i-1}} - \EE \mx{\tbX'} \right) \right] + \frac{(\nrho \wedge k)  \psi}{\rho} \right.\right.\\ &\qquad+
\left.\left.(\btheta'' - \tbtheta)^\top \sum_{i=\nrho+1}^{\nrho\vee k} \left[\left(  \EE \mx{\bX_{i-1}} - \EE \mx{\tbX''} \right)\right] - \frac{(k - \nrho)^+ \psi}{1-\rho} -n \psi \right]\right| \\
& \leq \max_{1 \leq k \leq n} \left| (\btheta' - \tbtheta)^\top \sum_{i=1}^{\nrho \wedge k} \left[\left(  \EE \mx{\bX_{i-1}} - \EE \mx{\tbX'} \right) \right]\right| \\
&\qquad + 
\max_{1 \leq k \leq n} \left| (\btheta'' - \tbtheta)^\top \sum_{i=\nrho+1}^{\nrho\vee k} \left[\left(  \EE \mx{\bX_{i-1}} - \EE \mx{\tbX''} \right)\right]\right| \\
& \qquad+ \left|\max_{1 \leq k \leq \nrho} \left(\frac{(\nrho \wedge k)}{\rho} - \frac{(k - \nrho)^+}{1-\rho} - n\right)\psi \right| .
\end{align*}
Because of Lemma \ref{geom}, the first two maxima are bounded, and the third one is obviously attained at $k = \nrho$ with value $\left(\frac{\nrho}{\rho}-n\right)\psi$, which is also bounded. Therefore the second term in \eqref{eq:max_ineq} is \ $\OO(1)$ \ as \ $n \to \infty$.
We can use Lemma \ref{altestconv} to show that the fourth term in \eqref{eq:max_ineq} is $\OO_{\PP}(\sqrt{n}),$ but for the first and third terms we also need another result. This can be found as Theorem 3.1. in \citeA{KoLe_98}. 

 \begin{Lem} \label{kole}
Let \ $(Y_n)_{n\in\NN}$ \ be a sequence of random variables with finite second moments, and let \ $(c_n)_{n \in \NN}$ \ be a sequence of nonnegative constants. Then, for any \ $a > 0$,
 \begin{align*}
  a^2 \PP \left( \max_{1 \leq k \leq n} c_k \left| \sum_{j=1}^k Y_j \right| 
					> a \right)
  &\leq \sum_{k=1}^{n-1} |c^2_{k+1} - c^2_k| \sum_{i,j=1}^k \EE(Y_iY_j) \\
  &\quad +2 \sum_{k=1}^{n-1} c_{k+1}^2 \left(\EE\left( Y^2_{k+1} \right) \sum_{i,j=1}^k \EE\left(Y_iY_j\right)\right)^{1/2} \\
  &\quad +2 \sum_{k=0}^{n-1} c_{k+1}^2 \EE(Y_{k+1}^2).
 \end{align*}
\end{Lem}

Lemma \ref{kole} can be applied to the fourth term in the following way to show that it is $\OO_{\PP}(n^{1-\gamma})$:

\begin{Lem}\label{kgamma}
For a time-homogeneous INAR($p$) process satisfying condition
$\textup{\textbf{C}}_0$ and any $\gamma < \frac{1}{4}$ we have
\[
\max_{1 \leq k \leq n} k^{\gamma-1}\vnorm{\sum_{i=1}^k (\bX_{k-1} - \EE(\bX_{k-1}))} = \OO_{\PP}(1).
\]
\end{Lem}

\begin{proof}
We will follow the proof of Lemma 4.2 in \citeA{HPS_07} and apply Lemma \ref{kole}
with \ $c_k = k^{\gamma-1}$ \ and \ $Y_{i,q} = X_{i-1-q} - \EE(X_{i-1-q})$ \ for \ $0 \leq q \leq p-1$
to show that the result holds for each component of the vectors.
This implies convergence of the 1-norm, and because of the equivalence of vector norms,
it is sufficient for the proof of the statement.
We have
\[
\left| \frac{1}{(k+1)^{2-2\gamma}} - \frac{1}{k^{2-2\gamma}} \right| \leq \frac{2(1-\gamma)}{k^{3-2\gamma}}
\]
and
\[
\sum_{i=1}^k\sum_{j=1}^k\EE(Y_{i,q}Y_{j,q}) = \sum_{i=-q}^{k-1-q}\sum_{j=-q}^{k-1-q} \cov(X_i,X_j) \leq \kappa k
\]
for some constant \ $\kappa$ \ according to Lemma \ref{vari}. Therefore,
\begin{align*}
&\sum_{k=1}^{n-1}\left| \frac{1}{(k+1)^{2-2\gamma}} - \frac{1}{k^{2-2\gamma}} \right| \sum_{i=1}^k\sum_{j=1}^k\EE(Y_{i,q}Y_{j,q})
     + 2 \sum_{k=1}^{n-1}k^{2\gamma-2} \EE^{1/2}(Y_{k+1,q}^2) \left(\sum_{i,j=1}^k \EE\left(Y_{i,q}Y_{j,q}\right)\right)^{1/2} \\
     & \qquad + 2 \sum_{k=0}^{n-1} k^{2\gamma - 2} \EE(Y_{k+1,q}^2) \\
	 & \leq (2\kappa-2\kappa\gamma)\sum_{k=1}^{n-1} k^{2\gamma-2} + 2 (\kappa U_1)^{1/2}  \sum_{k=1}^{n-1}k^{2\gamma-3/2} + 2 U_1 \sum_{k=0}^{n-1} k^{2 \gamma -2}, 
\end{align*}
where \ $U_1$ \ is the upper boundary of \ $(\var(X_n))_{n \in \NN}$. \ 
The limit of the right hand side as \ $n \to \infty$ \ is finite, which completes the proof. We note the necessarity of $\gamma < \frac{1}{4}$---otherwise the second term in the last expression would not be bounded. This indicates that if we would like to extend Theorem \ref{HA_mu_1} to $\gamma = \frac{1}{2}$ we need sharper estimates in place of Lemma \ref{kole}.
\end{proof}

For the estimation of the first term we will use Lemma \ref{kole} again:
\begin{Lem}
For a time-homogeneous INAR($p$) process satisfying condition
 $\textup{\textbf{C}}_0$ and any $\gamma \in (0,\frac{1}{4})$ we have
\[
\max_{1 \leq k \leq n} k^{\gamma-1}\left|\sum_{i=1}^k M_i\right| = \OO_{\PP}(1).
\]
\end{Lem}

\begin{proof}
We apply \ref{kole} in the same way as in the proof of Lemma \ref{kgamma} with $c_k = k^{\gamma-1}$ and $Y_i = M_i$. We note that the $M_k$ are martingale differences, therefore any product $M_iM_j, i \not = j$ has zero mean. Furthermore, the sequence $(\var M_k)_{k \in \NN}$ is clearly bounded, and denoting its upper bound by $U$, we have
\begin{align*}
&\sum_{k=1}^{n-1}\left| \frac{1}{(k+1)^{2-2\gamma}} - \frac{1}{k^{2-2\gamma}} \right| \sum_{i=1}^k\sum_{j=1}^k\EE(Y_{i}Y_{j})
     + 2 \sum_{k=1}^{n-1}k^{2\gamma-2} \EE^{1/2}(Y_{k+1}^2) \left(\sum_{i,j=1}^k \EE\left(Y_iY_j\right)\right)^{1/2} \\
     & \qquad + 2 \sum_{k=0}^{n-1} k^{2\gamma - 2} \EE(Y_{k+1}^2) \\
	 & \leq U (2-2\gamma)\sum_{k=1}^{n-1} k^{2\gamma-2} + 2U \sum_{k=1}^{n-1}k^{2\gamma-3/2} + 2U \sum_{k=0}^{n-1} k^{2 \gamma -2}, 
\end{align*}
whence the final steps of the proof are the same as in Lemma \ref{kgamma}.
\end{proof}

Summarizing our results for the terms of \eqref{eq:max_ineq}: the second term is $\OO(1)$, the fourth one is $\OO_{\PP}(n^{1/2})$ and the first and third terms are $\OO_{\PP}(n^{1-\gamma}),$ which completes our proof.
\qed

\subsection{Proof of Theorem \ref{HA_mu_2}}

The statement can be written in the form
 \[
   \lim_{K \to \infty} \sup_{n \in \NN}
    \PP( |\htau_n - \lfloor n \rho \rfloor| \geq K )
   = 0 ,
 \]
 which is equivalent to
 \[
   \lim_{K \to \infty} \limsup_{n \to \infty}
    \PP( |\htau_n - \lfloor n \rho \rfloor| \geq K )
   = 0 .
 \]
Hence to prove the statement it is enough to show that
 \begin{align}
  \lim_{K \to \infty} \limsup_{n \to \infty}
   \PP\left( \max_{\lfloor n \rho \rfloor - K < k < \lfloor n \rho \rfloor + K}
              \sum_{j=1}^k \hM_j^{(n)}
             \leq \max_{1 \leq k \leq \lfloor n \rho \rfloor - K} \sum_{j=1}^k \hM_j^{(n)}
      \right)
  &= 0 , \label{HA1} \\
  \lim_{K \to \infty} \limsup_{n \to \infty}
   \PP\left( \max_{\lfloor n \rho \rfloor - K < k < \lfloor n \rho \rfloor + K}
              \sum_{j=1}^k \hM_j^{(n)}
             \leq \max_{\lfloor n \rho \rfloor + K \leq k \leq n} \sum_{j=1}^k \hM_j^{(n)}
      \right)
  &= 0 . \label{HA2}
 \end{align}

For \eqref{HA1} we consider with a constant \ $K$,
 \ $K < \lfloor n \rho \rfloor$, \ the estimate
 \begin{multline*}
  \PP\left( \max_{\lfloor n \rho \rfloor - K < k < \lfloor n \rho \rfloor + K}
              \sum_{j=1}^k \hM_j^{(n)}
             \leq \max_{1 \leq k \leq \lfloor n \rho \rfloor - K}
                   \sum_{j=1}^k \hM_j^{(n)}
      \right) \\
  \begin{aligned}
  &\leq \PP\left( \sum_{j=1}^{\lfloor n \rho \rfloor} \hM_j^{(n)}
                  \leq \max_{1 \leq k \leq \lfloor n \rho \rfloor - K}
                        \sum_{j=1}^k \hM_j^{(n)}
           \right)
   = \PP\left( \min_{1 \leq k \leq \lfloor n \rho \rfloor - K}
                \sum_{j= k + 1}^{\lfloor n \rho \rfloor} \hM_j^{(n)}
               \leq 0
         \right) \\
  &= \PP\left( \min_{K \leq \ell \leq \lfloor n \rho \rfloor - 1}
                \sum_{j= \lfloor n \rho \rfloor - \ell + 1}^{\lfloor n \rho \rfloor}
                 \hM_j^{(n)}
               \leq 0
         \right) 
	= \PP\left( \min_{K \leq \ell \leq \lfloor n \rho \rfloor - 1} \ell^{-1}
                \sum_{j= \lfloor n \rho \rfloor - \ell + 1}^{\lfloor n \rho \rfloor}
                 \hM_j^{(n)}
               \leq 0
         \right).
  \end{aligned}
 \end{multline*}

We will use \eqref{decomp} again, note that the dominant term is the second one, and estimate the probability with this in mind. For any $K \leq \ell \leq \nrho$ the expression
\begin{equation}\label{eq:avdecomp}
\ell^{-1} \sum_{j= \lfloor n \rho \rfloor - \ell + 1}^{\lfloor n \rho \rfloor} \hM_j^{(n)}
\end{equation}
can be decomposed according to \eqref{decomp}. Now, \eqref{eq:avdecomp} can only be negative in two cases: either the second term in the decomposition is less or equal to $\frac{\psi}{2}$, or it is greater---in which case one of the other three terms has to be less than $-\frac{\psi}{6}$ (for the definition of $\psi$, see Theorem \ref{HA_mu_1}).

Now it is clear that, after applying \eqref{decomp} to \eqref{eq:avdecomp} we have
\begin{align*}
\PP&\left( \min_{K \leq \ell \leq \lfloor n \rho \rfloor - 1} \ell^{-1}
        \sum_{j= \lfloor n \rho \rfloor - \ell + 1}^{\lfloor n \rho \rfloor}
        \hM_j^{(n)}
        \leq 0
\right) \\
&\leq
	\PP\left( \min_{K \leq \ell \leq \lfloor n \rho \rfloor - 1} \ell^{-1}
			\sum_{j= \lfloor n \rho \rfloor - \ell + 1}^{\lfloor n \rho \rfloor}
			\left[(\balpha - \tbalpha)^\top\EE(\bX_{k-1}) + \left(\mu' - \tmu\right)\right]
			\leq \frac{\psi}{2}
	\right) \\
& \quad +
	\PP\left( \max_{K \leq \ell \leq \lfloor n \rho \rfloor - 1} \left|\ell^{-1}
			\sum_{j= \lfloor n \rho \rfloor - \ell + 1}^{\lfloor n \rho \rfloor}
			M_k\right|
			\geq \frac{\psi}{6}
	\right) \\
& \quad +
	\PP\left( \max_{K \leq \ell \leq \lfloor n \rho \rfloor - 1} \left|\ell^{-1}
			\sum_{j= \lfloor n \rho \rfloor - \ell + 1}^{\lfloor n \rho \rfloor}
			(\balpha - \hbalpha_n)^\top(\bX_{k-1} - \EE(\bX_{k-1}))\right|
			\geq \frac{\psi}{6}
	\right) \\
& \quad +
	\PP\left( \max_{K \leq \ell \leq \lfloor n \rho \rfloor - 1} \left|\ell^{-1}
			\sum_{j= \lfloor n \rho \rfloor - \ell + 1}^{\lfloor n \rho \rfloor}
			\left[(\tbalpha - \hbalpha_n)^\top\EE(\bX_{k-1}) + \left(\tmu - \hmu_n\right)\right]\right|
			\geq \frac{\psi}{6}
	\right).
\end{align*}

As a consequence of Lemma \ref{lem:asymperror} and \eqref{fergodic} the first term can be shown to converge to zero for any $K$ as $n \to \infty$ with the help of the following simple lemma:
\begin{Lem}\label{minconv}
Let \ $a_n \to a > 0, \ n \to \infty$ and \ $a_i > 0$ \ for all \ $i \in \NN$. \ Then
\[
\min_{1 \leq k \leq n} k^{-1} \sum_{i=n-k+1}^n a_i \to a, \quad n \to \infty.
\]
\end{Lem}
\begin{proof}
First we note that for any \ $\vare > 0$ and sufficiently large \ $n$, \ we have \ $\min_{1 \leq k \leq n} k^{-1} \sum_{i=n-k+1}^n a_i < a + \vare$. This can be seen by choosing \ $k=1$ \ for every \ $n$. \ Now we show \ $\min_{1 \leq k \leq n} k^{-1} \sum_{i=n-k+1}^n a_i > a - \vare$. Let \ $\nu(\vare)$ be the threshold index so that for \ $n > \nu(\vare)$ \ we have \ $|a_n - a| <  \frac{\vare}{2}$. \ Let us denote by \ $K$ \ the sum \ $\sum_{i=1}^{\nu(\vare)}a_i$. \ Clearly,
\[
\left|\min_{1 \leq k \leq n - \nu(\vare)} k^{-1} \sum_{i=n-k+1}^n a_i - a\right| < \frac{\vare}{2}.
\]
Furthermore, for any \ $n > k > n - \nu(\vare)$ \ we have
\[
k^{-1} \sum_{i=n-k+1}^n a_i > n^{-1} \sum_{i=\nu(\vare)}^n a_i = \frac{n-\nu(\vare)}{n} \left[(n-\nu(\vare))^{-1} \sum_{i=\nu(\vare)}^n a_i\right].
\]
For sufficiently large \ $n$ \ the first factor is close to $1$, \ and the second factor is closer to \ $a$ \ than \ $\frac{\vare}{2}$ for every $n$. \ This suffices for the proof.
\end{proof}

Because of \eqref{fergodic} and Lemma \ref{altestconv}, the fourth term also converges to zero for all \ $K$ \ as \ $n \to \infty$. Indeed, $(\tbalpha - \hbalpha_n) \stoch 0$ and 
\[\max_{K\leq \ell \leq \nrho}\ell^{-1}	\sum_{j= \lfloor n \rho \rfloor - \ell + 1}^{\lfloor n \rho \rfloor} \EE(\bX_{k-1})
	\leq \max_{1 \leq k \leq \nrho} \EE(\bX_k)\] 
	and due to \eqref{fergodic} the right hand side is bounded as $n \to \infty.$ The same reasoning applies to $(\tmu - \hmu_n).$ The convergence of the second and third terms is summarized in the following lemma.

\begin{Lem}\label{limsupk}
For a time-homogeneous INAR($p$) process satisfying condition
 $\textup{\textbf{C}}_0$ we have for any \ $a > 0$,
\[
\lim_{K \to \infty} \limsup_{n \to \infty} \PP\left( \max_{K \leq \ell \leq \lfloor n \rho \rfloor - 1}\left|
                \ell^{-1} \sum_{j= \lfloor n \rho \rfloor - \ell + 1}^{\lfloor n \rho \rfloor}
                 (\bX_{j-1} - \EE(\bX_{j-1}))\right|
               > a
         \right) = 0
\]
and
\[\lim_{K \to \infty} \limsup_{n \to \infty} \PP\left( \max_{K \leq \ell \leq \lfloor n \rho \rfloor - 1}\left|
                \ell^{-1} \sum_{j= \lfloor n \rho \rfloor - \ell + 1}^{\lfloor n \rho \rfloor}
                 M_j\right|
               > a
         \right) = 0.
\]
\end{Lem}
 
\begin{proof} 
Similarly to the proof of Lemma \ref{kgamma} we will again employ Lemma \ref{kole} with $c_k = (K+k-1)^{-1}$ and $Y_{1,q} = \sum_{j= \lfloor n \rho \rfloor - K + 1}^{\lfloor n \rho \rfloor} X_{j-q}$
and $Y_{i,q} = X_{\lfloor n \rho \rfloor - K +1 - i - q}$ for $i \geq 2$ and $0 \leq q \leq p-1$.

By an easy calculation
\[
\sum_{i,j=1}^k\EE(Y_iY_j) = \sum_{i,j=\lfloor n \rho \rfloor - K - k + 1}^{\lfloor n \rho \rfloor} \EE((\bX_{i-1} - \EE(\bX_{i-1}))(\bX_{j-1} - \EE(\bX_{j-1}))).
\]
Therefore, applying the same estimations and notations as in the proof of Lemma \ref{kgamma} with $\gamma = 0$, we obtain the following upper limit for the probability in question:
\[
2\kappa\sum_{\ell=K}^{\nrho-1} (\ell+1)^{-2} + U_1 \sum_{\ell=K}^{\nrho-1}(\ell+1)^{-3/2} + \frac{U_1}{K}+ U_1 \sum_{\ell=K-1}^{\nrho-1} (\ell+1)^{-2}.
\]
It is obvious that as \ $n \to \infty$ \ and then \ $K \to \infty$, the above expression converges to 0, which suffices for our proof. For the second statement the arguments are the same. We note that \ $(M_n)_{n \in \NN}$ \ is a martingale difference sequence, hence its elements are pairwise uncorrelated. Furthermore,\ $\var(M_n)_{n \in \NN}$ \ is bounded, which implies \ $\var(M_1 + \ldots + M_n) = \OO(n)$ \ immediately.
\end{proof}

To prove \eqref{HA2} the proof is analogous wiht one exception: in place of Lemma \ref{minconv} we need the following result.
 \begin{Lem}
	 Let \ $a_n \to a > 0, \ n \to \infty$ and \ $a_i > 0$ \ for all \ $i \in \NN$. \ Then
	 \[
	 \lim_{K \to \infty} \inf_{k \geq K} k^{-1} \sum_{i=1}^k a_i = a, \quad n \to \infty.
	 \]
 \end{Lem}
 
 \begin{proof}
 We only need to observe that convergence of $a_n$ implies convergence in Cesaro mean as well, therefore, for a sufficiently large $K$ and for all $k \geq K$ the average $k^{-1} \sum_{i=1}^k a_i$ is close to $a$.
 \end{proof}

\subsection{Adapting the proofs of Theorem \ref{HA_mu_1} and Theorem \ref{HA_mu_2} to other forms of the alternative hypothesis}\label{adapt}

The proofs of Theorem \ref{HA_mu_1} and Theorem \ref{HA_mu_2} heavily exploit the relatively simple structure of the test process for detecting change in \ $\mu$ \ only. For the other parameters, even the limit in Theorem \ref{HA_mu_1} will be different, namely, we have the following theorem.
\begin{Thm}
Suppose that \ $\mathrm{H}_\mathrm{A}$ \ holds with
 \ $\tau = \max(\lfloor n \rho \rfloor,1)$, \ $\rho \in \left(0, 1\right)$, \ and the change is
 only in $\alpha_q$, namely, $\alpha_q$ changes from
 \ $\alpha_q'$ \ to \ $\alpha_q''$, \ where \ $\alpha' > \alpha'' > 0$. Suppose, furthermore, that condition $\textup{\textbf{C}}_\textup{A}$ holds.
\ Then for any \ $\gamma \in \left(0, \frac{1}{4}\right)$ \ we have
 \[
   \max_{1 \leq k \leq n} \sum_{j = 1}^k \hM_j^{(n)}X_{j-q}
   = n \psi_q
     + \OO_{\PP}(n^{1-\gamma}) \qquad \text{as \ $n \to \infty$,}
 \]
 with
 \[
   \psi_q := \rho (1-\rho) (\alpha_q'-\alpha_q'')
     \be_{q}^\top
     \bC'' \tbQ^{-1} \bC'
     \be_{q} > 0 ,
 \]
 where \ $\be_{q}$ \ is the \ $q$-th \textup{(}$p+1$-dimensional\textup{)}
 unit vector.
\end{Thm}

The proof is analogous to that of Theorem \ref{HA_mu_1}. All the techniques in the proofs can be directly adapted, noting that all means converge due to \eqref{fergodic}, and Lemma \ref{altestconv} will remain true under any form of the alternative hypothesis. In place of \eqref{decomp}, we can then write
\begin{equation}\label{alphadecomp}
\begin{split}
\hM_k^{(n)}X_{k-q} & = M_k'X_{k-q} + \left[(\balpha' - \tbalpha)^\top\EE(\bX_{k-1}X_{k-q}) + \left(\mu' - \tmu\right)\EE(X_{k-q})\right] \\
& \quad + (\balpha' - \hbalpha_n)^\top(\bX_{k-1}X_{k-q} - \EE(\bX_{k-1}X_{k-q})) + \left(\mu' - \tmu\right)(X_{k-q} - \EE(X_{k-q})) \\
& \quad + \left[(\tbalpha - \hbalpha_n)^\top\EE(\bX_{k-1}X_{k-q})) + \left(\tmu - \hmu_n\right)\EE(X_{k-q})\right].
\end{split}
\end{equation}

As in \eqref{decomp}, the second term will be dominant here and in place of Lemma \ref{lem:asymperror} we have
\[
(\btheta' - \tbtheta)^\top
\begin{bmatrix}
\EE (\tbX' \tX'_{-q+1}) \\
\EE (\tX'_{-q+1})
\end{bmatrix}
= \frac{\psi_q}{\rho} > 0, \qquad (\btheta'' - \tbtheta)^\top \begin{bmatrix}
\EE (\tbX'' \tX''_{-q+1}) \\
\EE (\tX''_{-q+1})
\end{bmatrix}
= -\frac{\psi_q}{1-\rho},
\]
We need (ii) in Lemma \ref{vari} to be able to apply Lemma \ref{kole} to the first and third terms in \eqref{alphadecomp}. The rest of the adaptation is straightforward.

 If there is a change in multiple parameters, the analogue of Theorem \ref{HA_mu_1} can still be stated, but the sign of the dominant term depends nontrivially on the directions of the changes in these parameters (i.e., in Lemma \ref{lem:asymperror} we can determine the limit but we cannot determine its sign without calculating it explicitly). The one-sided test, however, relies on our knowledge of the sign of the dominant term. This warns us that we should use the one-sided tests only when testing for change in a single parameter.
 
\section{Technical details}

\subsection{Invertibility of the matrices $\bQ_n, \ \bC'$ and $\bC''$}\label{sec:invert}

In \eqref{CLSests} we assumed that the matrix $\bQ_n$ is invertible, and similarly, in \ref{altests} we assumed that $\bC'$ and $\bC''$ are positive definite. The following two lemmas will show that these assumptions are correct.

\begin{Lem}\label{lem:qinvert}
For a homogeneous INAR($p$) process with $\mu>0$, for which either $\alpha_q \in (0,1)$ for some $q \in \{1,2,\ldots,p\}$ or $\sigma>0$, we have
\[
\PP(\bQ_n \text{ is singular}) \to 0.
\]
\end{Lem}

\begin{proof}
Since
\[
\bQ_n = \sum_{i=1}^{n}\mx{\bX_{i-1}} \mx{\bX_{i-1}}^\top
\]
is a sum of positive semidefinite matrices, it is positive semidefinite itself. Therefore, its singularity is equivalent to the condition that for some $0 \not = \bv \in \RR^{p+1}$ and every index $i\in\{1,\ldots,n\},$ we have
\[
\bv^\top \mx{\bX_{i-1}} \mx{\bX_{i-1}}^\top \bv = 0,
\]
which is equivalent to the condition that the linear span of \[\left\{\mx{\bX_{i-1}}, i=1,\ldots,n\right\}\] is a proper subspace of $\RR^{p+1}$. Now, using the continuity of probability, our statement is equivalent to the following:
\begin{equation}\label{eq:span}
\PP \left(\text{span} \left\{\mx{\bX_{i-1}}, i \in \NN \right\} < \RR^{p+1} \right) = 0,
\end{equation}
where $<$ denotes proper subspace.
For simplicity, throughout the proof we will use the notation
\[
\bY_i = \mx{\bX_{i-1}}, \  i \in \NN.
\]
It is clear that all values of $(\bY_i)_{i \in \NN}$ fall into $\NN_0^p \times \{1\}$. We introduce the following notation for the set of spaces that can be spanned by the values of the process:
\begin{equation}
\cS:=\{S < \RR^{p+1}: S = \text{span}\{\by_1, \by_2, \ldots, \by_n\},  \ \by_1, \by_2, \ldots, \by_n \in \NN_0^p\times\{1\}, \ n \in \NN\}.
\end{equation}
We can notice that $\cS$ is countable. Indeed, every generating system of a subspace contains a basis, therefore every subspace $S \in \cS$ has a basis whose elements are from $\NN_0^p\times \{1\}$. Such a basis is from $(\NN_0^p\times \{1\})^{k}$ where $k = \dim S$, and $0 \leq k \leq p$, and of course a basis corresponds to only one subspace. Now, since $\NN_0^p\times \{1\}$ is countable, $(\NN_0^p\times \{1\})^{k}$ is also countable for any $k \in \NN$, therefore $\cup_{k=0}^p (\NN_0^p\times \{1\})^{k}$ is also countable, and so is $\cS$.

Now we reformulate the event in \eqref{eq:span}:
\begin{equation}
\begin{split}
\left\{\text{span} \left\{\bY_i, i\in \NN \right\} < \RR^{p+1} \right\} &=
\bigcup_{S < \RR^{p+1}} \left\{\text{span} \left\{\bY_i, i \in \NN \right\} = S \right\} =
\bigcup_{S \in \cS} \left\{\text{span} \left\{\bY_i, i \in \NN \right\} = S \right\} \\& \subseteq \bigcup_{S \in \cS} \left\{\text{span} \left\{\bY_i, i \in \NN \right\} \subseteq S \right\}.
\end{split}
\end{equation}
Since the last union is countable, we can apply $\sigma$-subadditivity to show \eqref{eq:span} if we can prove
\begin{equation}\label{eq:impossubsets}
\PP \left(\text{span}\left\{\bY_i, i \in \NN \right\} \subseteq S\right) = \lim_{n \to \infty} \PP \left(\text{span}\left\{\bY_1, \bY_2, \ldots, \bY_n \right\} \subseteq S\right) = 0, \quad \forall S \in \cS.
\end{equation}
Here the first equality is trivial by the continuity of probability; it is the second equality which requires a more detailed proof.
The first step in the proof of \eqref{eq:impossubsets} relies on the mechanism by which the components of $\bY_{i+1}$ can be obtained from those of $\bY_i$. For a fixed $S \in \cS$ the elements of $\cS$ can be viewed as the solutions of a homogeneous system of independent linear equations, i.e., $\by \in S$ if and only if
\begin{equation}\label{eq:LES}
\sum_{j=1}^{p+1} \lambda_{i,j} y^{(j)} = 0, \quad i=1,2,\ldots , p+1-\dim S.
\end{equation}
This representation is not unique, but we can fix one such representation. Now let us introduce
\begin{align*}
K(S) &:= \min\{j \in \{1,2,\ldots,p\}:\max_i |\lambda_{i,j}| > 0\} \text{ and } \\
c(s) &:=\min \{i\in \{1,2,\ldots, p+1-\dim S\}: |\lambda_{i,K(S)}| > 0\}.
\end{align*}
This notation means that $K(S)$ is the first column index for which a nonzero coefficient appears in some equation in \eqref{eq:LES} and the first nonzero coefficient in the $c(s)$-th equation has index $K(S)$.
We also note that $K(S) = p+1$ is impossible because that would mean that the only equation is $y^{(p+1)}=0$, which does not hold for any element of $\NN_0^p \times \{1\}.$

Let us now fix an arbitrary $i \in \NN$ and $\omega \in \Omega$ from our underlying probability space such that $\bY_i(\omega) = \by =(y^{(1)}, y^{(2)}, \ldots, y^{(p)}, 1)^\top$. Then we have 
\[
\bY_{i+K(S)}(\omega) = \left(X_{i+K(S)-1}(\omega), \ldots, X_i(\omega), y^{(1)}, \ldots, y^{(p-K(S))},1\right)^\top \quad \text{(see \eqref{eq:stochregression})}.
\]
Hence, for $\bY_{i+K(S)}(\omega) \in S$ to hold, it is necessary (but usually not sufficient) that $\bY_{i+K(S)}(\omega)$ satisfy the $c(s)$-th equation in \eqref{eq:LES}, i.e.,
\begin{align*}
\sum_{j=1}^{K(S)} &\lambda_{c(s),j} X_{i+K(S)-j}(\omega) + \sum_{j=K(S)+1}^{p} \lambda_{c(s),j} y^{(j-K(S))} + \lambda_{c(s),p+1} = 0 \\ &\Leftrightarrow \lambda_{c(s), K(S)} X_i(\omega) + \sum_{j=K(S)+1}^{p} \lambda_{c(s),j} y^{(j-K(S))} + \lambda_{c(s),p+1}= 0.
\end{align*}
This linear equation has a unique solution for $X_i(\omega)$ because $\lambda_{c(s), K(S)} \not = 0$. Let us denote this unique solution by $m(\by, S)$ (by simple algebraic considerations one can see that this quantity does not depend on the representation on \eqref{eq:LES}, but this is not necessary to our proof). Therefore, if $X_i(\omega) \not = m(\by,S),$ then $\omega \not \in \{\bY_{i+K(S)} \in S, \bY_i = \by\}$, hence
\[
\{\bY_{i+K(S)} \in S, \bY_i = \by\} \subseteq \{X_{i} = m(\by,S), \bY_i = \by\} \quad \forall i \in \NN, \quad \forall \by \in \NN_0^p \times \{1\}.
\]
If $m(\by, S) \not \in \NN$, then we have $\{\bY_{i+K(S)} \in S, \bY_i = \by\} = \emptyset$.

Now we will consider the second event in \eqref{eq:impossubsets} for $n=n+K(S)$, and split it according to the initial value of the process:
\begin{equation}\label{eq:split}
\left\{\text{span}\left\{\bY_1, \bY_2, \ldots, \bY_{n+K(S)} \right\} \subseteq S\right\} = \bigcup_{\by_1 \in \NN_0^p \times \{1\}} \{\text{span}\left\{\bY_1, \bY_2, \ldots, \bY_{n+K(S)}\right\} \subseteq S, \bY_1 = \by_1 \}.
\end{equation}
The individual events in the union can be transformed in the following way: 
\begin{equation}\label{eq:equiv}
\begin{split}
&\{\text{span}\left\{\bY_1, \bY_2, \ldots, \bY_{n+K(S)}\right\} \subseteq S, \bY_1 = \by_1 \} = 
\{\bY_1 \in S, \bY_2 \in S, \ldots, \bY_{n+K(S)} \in S, \bY_1=\by_1\} \\
&\quad = \{\bY_1 \in S, \bY_2 \in S, \ldots, \bY_{K(S)} \in S, \bY_1 = \by_1, \bY_{1+K(S)}\in S, \ldots, \bY_{n+K(S)} \in S\} \\
&\quad \subseteq \{\bY_1 \in S, \bY_2 \in S, \ldots, \bY_{K(S)} \in S, \bY_1 = \by_1, X_1=m(y_1, S), \bY_{2+K(S)} \in S, \ldots, \bY_{n+K(S)} \in S\} \\
&\quad = \{\bY_1 \in S, \bY_2 \in S, \ldots, \bY_{K(S)} \in S, \bY_1 = \by_1, \bY_2=\by_2, \bY_{2+K(S)} \in S, \ldots, \bY_{n+K(S)} \in S\} \\
&\quad \subseteq \{\bY_1 \in S, \bY_2 \in S, \ldots, \bY_{K(S)} \in S, \bY_1 = \by_1, \bY_2=\by_2, X_2=m(y_2,S), \bY_{3+K(S)} \in S, \ldots, \bY_{n+K(S)} \in S\} \\
&\quad = \{\bY_1 \in S, \bY_2 \in S, \ldots, \bY_{K(S)} \in S, \bY_1 = \by_1, \bY_2=\by_2, \bY_3=\by_3, \bY_{3+K(S)} \in S, \ldots, \bY_{n+K(S)} \in S\}\\
&\quad \ \vdots \\
&\quad \subseteq \{\bY_1 \in S, \bY_2 \in S, \ldots, \bY_{K(S)} \in S, \bY_1 = \by_1, \bY_2=\by_2, \bY_3=\by_3, \ldots, \bY_{n} =\by_n\},
\end{split}
\end{equation}
where the sequence $(\by_i)_{i=1}^n$ is defined by the recursion
\begin{equation}\label{eq:ysequence}
\by_i =
\begin{bmatrix}
m(\by_{i-1},S) \\
y_{i-1}^{(1)} \\
\vdots \\
y_{i-1}^{(p-1)} \\
1
\end{bmatrix}
, \quad i=2,3,\ldots,n.
\end{equation}
We would like to represent the probability of the last event in \eqref{eq:equiv} as a product of transition probabilities. For this we first need to determine whether the event is empty, and now we will give to necessary conditions on $\by_1$ for its nonemptiness.
The first condition is, clearly, that all elements of the sequence defined in \eqref{eq:ysequence} fall into $\NN_0^p \times \{1\}$. We will not investigate this condition in any further detail, we only note that this imposes a deterministic condition on $\by_1$. Another deterministic condition is that $\bY_1 \in S, \bY_2 \in S, \ldots, \bY_{K(S)} \in S$ should all hold. Because the first $K(S)-1$ coefficients are all zero in any equation in $\eqref{eq:LES}$ and $\bY_1$ contains all the components indexed $K(S)$ or greater in $\bY_1, \ldots, \bY_{K(S)}$, the validity of these inclusions is determined by $\by_1$ alone. This imposes the second (again, deterministic) condition on $\by_1$. If we denote the set of $\by_1$ which fulfill both these conditions by $U_n$, we have from \eqref{eq:split} and \eqref{eq:equiv},
\[
\left\{\text{span}\left\{\bY_1, \bY_2, \ldots, \bY_{n+K(S)} \right\} \subseteq S\right\} \subseteq
\bigcup_{\by_1 \in U_n} \{\bY_1 = \by_1, \bY_2=\by_2, \bY_3=\by_3, \ldots, \bY_{n} =\by_n\},
\]
hence by $\sigma$-subadditivity ($U_n$ is clearly countable),
\begin{equation}
\begin{split}
\PP\left(\text{span}\left\{\bY_1, \bY_2, \ldots, \bY_{n+K(S)} \right\} \subseteq S\right) &\leq
\sum_{\by_1 \in U_n} \PP \left(\bY_1 = \by_1, \bY_2=\by_2, \ldots, \bY_{n} =\by_n\right) \\
& = \sum_{\by_1 \in U_n} \PP(\bY_1=\by_1) p_{\by_1,\by_2}p_{\by_2,\by_3}\cdots p_{\by_{n-1}, \by_{n}},
\end{split}
\end{equation}
where $p_{\bu,\bv}$ denotes the transition probability of $\bY$ from $\bu$ to $\bv$. 
Because the sets $(U_n)_{n \in \NN}$ form a nonincreasing sequence (the second condition does not depend on $n$, and the first one become more restrictive as $n$ increases), it is sufficient to show that for any sequence $(\by_i)_{i \in \NN} \in (\NN_{0} ^p \times \{1\})^{\NN}$ we have
\begin{equation}\label{eq:fixtrajectory}
\lim_{n \to \infty} p_{\by_1,\by_2}p_{\by_2,\by_3}\cdots p_{\by_{n-1}, \by_{n}} = 0,
\end{equation}
and from this we will get \eqref{eq:impossubsets}.
For the proof of \eqref{eq:fixtrajectory} we will need to establish upper bounds for the transition probabilities.
We will first consider the case when $\sigma > 0$, i.e., when the innovation distribution is nondegenerate.

Let us fix $\bu, \bv \in \NN_{0} ^p \times \{1\}$ so that $v^{(2)} = u^{(1)}, v^{(3)} = u^{(2)}, \ldots, v^{(p)} = u^{(p-1)}, v^{(p+1)}=1$. We would like to give an upper bound for $p_{\bu,\bv}$. We have for every $i \in \NN$ and any $m \in \NN_0$,
\begin{align*}
\PP\left(\bY_{i+1} = \bv \big|\bY_i=\bu, \sum_{j=1}^p \sum_{\ell=1}^{u^{(j)}}\xi_{j,i,\ell} =m \right) &= \PP\left(\vare_{i} = v^{(1)} - m \big|\bY_i=\bu, \sum_{j=1}^p \sum_{\ell=1}^{u^{(j)}}\xi_{j,i,\ell} =m\right) \\ &= \PP\left(\vare_{i} = v^{(1)}- m\right).
\end{align*}
Applying the law of total probability we get
\begin{equation}
p_{\bu,\bv} = \sum_{m \in \NN_0} \PP(\vare_{i} = v^{(1)}- m) \PP\left(\sum_{j=1}^p \sum_{\ell=1}^{u^{(j)}}\xi_{j,i,\ell} =m\right) \leq \max_{k \in \NN_0} \PP(\vare_{i} = k) < 1,
\end{equation}
since the innovation distribution was nondegenerate. Therefore, if $\sigma>0$, then we have a uniform upper bound on the transition probabilities, which implies $\eqref{eq:fixtrajectory}$ immediately.

The other case is if the innovation distribution is degenerate. First we note that in this case the innovation is equal to its expectation $\mu > 0$ almost surely, so that all components of $\bY_i$ are  positive for $i \geq p+1$. According to the conditions, there is a coefficient $\alpha_q, q \in \{1,2,\ldots,p\}$ such that $0 < \alpha_q < 1$. Similarly to the previous reasoning, if additionally we suppose that all components of $\bu$ and $\bv$ are greater or equal to $\mu$, we have
\begin{align*}
\PP&\left(\bY_{i+1} = \bv \big|
		\bY_i=\bu, \mu+\left(\sum_{\substack{j=1 \\ j \not = q}}^p 
		\sum_{\ell=1}^{u^{(j)}}\xi_{j,i,\ell}\right) + \sum_{\ell=1}^{u^{(q)}-1} \xi_{q,i,\ell} =m \right)
	\\ &= \PP\left(\xi_{q,i,u^{(q)}} = v^{(1)} - m \big|
		\bY_i=\bu, \mu+\left(\sum_{\substack{j=1 \\ j \not = q}}^p 	
		\sum_{\ell=1}^{u^{(j)}}\xi_{j,i,\ell}\right) + \sum_{\ell=1}^{u^{(q)}-1} \xi_{q,i,\ell} =m\right)
	\\ &= \PP\left(\xi_{q,i,u^{(q)}} = v^{(1)}- m\right).
\end{align*}
Here we note that $\xi_{q,i,u^{(q)}}$ is a meaningful notation because $u^{(q)} \geq \mu$ and $\mu$ is a positive integer.
Applying the law of total probability again, we have that
\[
p_{\bu,\bv} \leq \max(\alpha_q, 1-\alpha_q) < 1,
\]
which again gives a uniform upper bound for the transition probabilities and yields \eqref{eq:fixtrajectory}. With this our proof is complete.
\end{proof}

It may be worth noting that Lemma \ref{lem:qinvert} imposes very weak conditions on the process---we only neglect the trivial case when all innovation and offspring distributions are degenerate. Also, the lemma does not require that the process start from zero---the initial distribution can be arbitrarily chosen on $U$. This gives us a chance to prove two important corollaries.

\begin{Cor}
For an INAR($p$) process under the alternative hypothesis satisfying the assumptions of Lemma \ref{lem:qinvert} both before and after the change, and $\tau = \max(1, \nrho)$ for some $\rho > 0$ constant, we have
\[
\PP(\bQ_n \text{is singular}) \to 0.
\]
\end{Cor}

\begin{proof}
To show this statement we only need to note that due to Lemma \ref{lem:qinvert} we have
\[
\PP(\bQ_{\nrho} \text{is singular}) \to 0,
\]
and clearly
\[
\{\bQ_n \text{is singular}\} \subseteq \{\bQ_{\nrho} \text{is singular}\}
\]
due to the reasoning at the beginning of the proof of Lemma \ref{lem:qinvert}.
\end{proof}

\begin{remark}
The conditions of this corollary could be greatly relaxed. Actually, it can be shown to be true for any change point sequence $\tau_n$, but a proper formulation is tedious and not a goal of this paper.
\end{remark}

\begin{Cor}
Under the conditions of Theorem \ref{HA_mu_1} both \ $\bC'$ \ and \ $\bC''$ \ are positive definite.
\end{Cor}

\begin{proof}
First we prove for $\bC'$. We note that Lemma \ref{lem:qinvert} did not impose any conditions on the initial distribution of the process $\bY$, therefore we can start the process from its stationary distribution (the existence of which is a trivial corollary of the existence of such a distribution for $\bX$ before the change). Now, the singularity of $\bC'$ is equivalent to the condition 
\[
\PP\left(\mx{\tX'} \in S\right) < 1, \quad \forall S < \RR^{p+1}.
\]
Let us now suppose that the stationary distribution of $\bY$ is concentrated on a proper subspace $S < \RR^{p+1}$. From \eqref{eq:impossubsets}, however, we conclude that the probability of the process remaining in $S$ forever is zero. As the distribution of $\bY_n$ is the stationary distribution for every time $n$, this is an immediate contradiction. Therefore $\bC'$ is nonsingular, but since it is a covariance matrix, it is positive semidefinite, therefore it has to be positive definite. The proof is the same for $\bC''$.
\end{proof}

\subsection{The conditional moments of $M_k$}
\label{cmoments}
We shall now derive several moments of $M_k$ conditionally on $\cF_{k-1}$ (this calculation is a reproduction of that in \citeA{TSz_OTDK}). Let us write $M_k$ in the form
\[
M_k = \sum_{j=1}^{X_{k-1}}(\xi_{1,k,j} - \alpha_1) + \sum_{j=1}^{X_{k-2}}(\xi_{2,k,j} - \alpha_2) + \ldots + \sum_{j=1}^{X_{k-p}}(\xi_{p,k,j} - \alpha_p) + (\vare_k - \mu). 
\]
All the terms on the right hand side have zero mean and are independent of each other conditionally on $\cF_{k-1}$, therefore
\[
\EE(M^2_k|\cF_{k-1}) = \alpha_1(1-\alpha_1)X_{k-1} + \ldots + \alpha_p(1-\alpha_p) X_{k-p} + \sigma^2.
\]
Similarly,
\begin{align*}
\EE(M^4_k|\cF_{k-1}) & = \sum_{i=1}^p \EE((\xi_{i,1,1} - \alpha_i)^4) X_{k-i} + 3 \sum_{i,j=1, i \not = j}^{p}\EE((\xi_{i,1,1} - \alpha_i)^2(\xi_{j,1,1} - \alpha_j)^2)X_iX_j \\
 & \quad + 6 \sum_{i=1}^p \binom{X_i}{2} \EE^2((\xi_{i,1,1} - \alpha_i)^2) + 6 \sum_{i=1}^p X_{k-i} \EE((\xi_{i,1,1} - \alpha_i)^2(\vare_1 - \mu)^2) + \EE((\vare_1 - \mu)^4) \\
 & = \balpha^\top_4 X_{k-i} + 3 \sum_{i,j=1, i \not = j}^{p}\alpha_i(1-\alpha_i)\alpha_j(1-\alpha_j)X_iX_j \\
 & \quad + 6 \sum_{i=1}^p \binom{X_i}{2} \alpha_i^2(1-\alpha_i)^2 + 6 \balpha^\top_2\bX_{k-1} \sigma^2 + \EE((\vare_1 - \mu)^4).
\end{align*}

\section*{Acknowledgements}

The authors wish to thank Edit Gombay at the University of Alberta and Lajos Horv\'ath at the University of Utah for drawing their attention to this research topic and directing them towards some very useful literature.

We are deeply grateful to several anonymous reviewers for a thorough analysis of the manuscript and insightful comments which greatly improved the paper.

The research was partially supported by the Hungarian Scientific Research Fund under Grant No. OTKA T-079128, the Hungarian--Chinese Intergovernmental S \& T Cooperation Programme for 2011-2013 under Grant No. 10-1-2011-0079 and the T\'AMOP-4.2.2/B-10/1-2010-0012 project.

\bibliographystyle{apacite}
\bibliography{change_in_inar(p)_03}

\end{document}